%% file: main.tex
\documentclass{article}
\title{Polarizations on a triangulated category}
\renewcommand\footnotemark{}
\thanks{2020 {\em Mathematics Subject Classification.} 14F08 (primary) 14E05, 18G80, 53D37 (secondary)}\thanks{{\em Key words and phrases.} Balmer spectrum, homological mirror symmetry, Matsui spectrum, noncommutative projective geometry, perfect derived category, triangulated category, tensor triangulated category}

\input{name}
\input{style}

\begin{document}
\maketitle
\begin{abstract} 
In a recent collaboration, Hiroki Matsui and the author introduced a new proof of the reconstruction theorem of Bondal--Orlov and Ballard, using Matsui's construction of a ringed space associated to a triangulated category. This paper first shows that these ideas can be applied to reconstructions of more general varieties from their perfect derived categories. For further applications of these ideas, we introduce the framework of a polarized triangulated category, a pair $(\mathcal T,\tau)$ consisting of a triangulated category $\mathcal T$ and an autoequivalence $\tau$ (called a polarization), to which we can associate a ringed space called the pt-spectrum. As concrete applications, we observe that several reconstruction results of Favero naturally fit within this framework, leading to both generalizations and new proofs of these results. Furthermore, we explore broader implications of polarizations and pt-spectra in tensor triangular geometry, noncommutative projective geometry, birational geometry and homological mirror symmetry.

\end{abstract}
\tableofcontents

\section{Introduction}
In \cite{ito2024new}, the authors provide a new proof of the reconstruction theorem of Bondal--Orlov and Ballard, which asserts that we can reconstruct a Gorenstein projective variety $X$ with (anti)ample canonical bundle $\omega_X$ from its perfect derived category $\perf X$ and moreover that if there is an equivalence $\perf X \simeq \perf Y$ for a projective variety $Y$, then we have $X \iso Y$. This result was originally established for smooth varieties by Bondal--Orlov \cite{bondal_orlov_2001} and later extended to Gorenstein cases by Ballard \cite{ballard2011derived}. A key ingredient of the new proof in \cite{ito2024new} is Matsui's construction of a ringed space associated to a triangulated category.

In this paper, we extend these techniques to a more general setting (Theorem \ref{thm: intro1}). This generalization simultaneously encompasses the reconstruction results of Bondal--Orlov and Ballard as well as the reconstruction of elliptic curves from their perfect derived categories (cf. \cite{huybrechts_2016}*{p.134}). Moreover, while the original theorem assumes projectivity (due to the (anti)ampleness condition on $\omega_X$), we demonstrate that properness is sufficient thanks to a recent work of Matsukawa \cite{matsukawa2025spectrum}. In particular, in the non-projective proper setting, the Proj of the canonical ring cannot give back the original variety, which was one of the key points in the original Bondal--Orlov reconstructions. Moreover, our set-up aligns with the fact that projectivity is not a property of the derived category in the sense that there are a projective variety and a non-projective proper variety whose derived categories are equivalent (e.g., when the latter is a certain non-projective flop of the former).

Next, we observe that the key ingredients underlying these proofs are the role of the Serre functor and the fact that any equivalence commutes with Serre functors. As a consequence, it turns out that it is natural to work with a new framework: a \textbf{polarized triangulated category}, defined as a pair $(\cal T, \tau)$, where 
$\cal T$ is a triangulated category equipped with an autoequivalence $\tau$ (which we call a \textbf{polarization}). This perspective unifies and extends previous reconstruction results and leads to further generalizations, including variants of the work of Favero \cite{FAVERO20121955} (Theorem \ref{intro:thm3}).

Finally, we explore broader applications of this framework. In particular, we show that polarized triangulated categories provide new insights into noncommutative projective geometry (in the sense of \cite{artin1994noncommutative}), Iitaka fibrations in birational geometry, and homological mirror symmetry.

Now, to explain these results in more detail, let us begin with a recap of \cite{ito2024new}. 
\begin{itemize}
    \item For any variety $X$, we can construct a ringed space $\spec_\vartriangle \perf X$, called the \textbf{Matsui spectrum}, using only the triangulated category structure of $\perf X$. Points in $\spec_\vartriangle \perf X$ correspond to certain thick subcategories of $\perf X$.
    \item For any variety $X$, we can construct a locally ringed space $\spec_{\tens_X} \perf X$, called the \textbf{Balmer spectrum}, using additional data of the monoidal structure $\tens_X:=\tens_{\ecal O_X}^\bb L$ on the triangulated category $\perf X$. It is known that we have an isomorphism $$\spec_{\tens_X^\bb L} \perf X \iso X$$ of (locally) ringed spaces. Moreover, if $X$ is quasi-projective, we can show the Balmer spectrum is an open ringed subspace of the Matsui spectrum:
    \[
    X \iso \spec_{\tens_X} \perf X \underset{\mathrm{open}}{\subset} \spec_\vartriangle \perf X.
    \]
    \item In the sequel, let us assume $X$ is a projective Gorenstein variety so that $\omega_X$ is a line bundle and $\perf X$ admits the Serre functor $\bb S := - \tens_{X} \omega_X[\dim X]$. Now, let us consider the subspace
    \[
    \spec^\ser \perf X := \{\cal P \in \spec_\vartriangle \perf X \mid \bb S(\cal P) = \cal P\} \subset \spec_\vartriangle \perf X,
    \]
    which we call the \textbf{Serre invariant locus}. 
    \item When $\omega_X$ is (anti)ample, we can show 
    \[
    \spec^\ser \perf X = \spec_{\tens_X} \perf X \underset{\mathrm{open}}{\subset} \spec_\vartriangle \perf X
    \]
    as ringed spaces, where the ringed space structure of the Serre invariant locus is given by the restriction of the structure sheaf of the Matsui spectrum. In particular, we have an isomorphism $X \iso \spec^\ser \perf X$, where the right hand side only uses the triangulated category structure of $\perf X$. Therefore, we have reconstructed the variety $X$ from the triangulated category structure of $\perf X$. 

    \item Note any exact equivalence $\perf X \simeq \perf Y$ induces an isomorphism 
    \[
    X \iso \spec^\ser \perf X \iso \spec^\ser \perf Y
    \]
    since an exact equivalence necessarily commutes with the Serre functors. Now, as in the second item above, we have that $\spec_{\tens_Y} \perf Y \iso Y$ is an open subscheme of $\spec^\ser \perf Y  \iso X$ while $Y$ is also a closed subscheme of $X$ as $Y$ is proper and $X$ is separated. Therefore, we must have $X \iso Y$ as $X$ is connected.
\end{itemize}
It is worth noting that the reconstruction theorems by Bondal--Orlov and Ballard are originally proven by identifying skyscraper sheaves and line bundles via homological methods including spectral sequences while our arguments above rely on topological properties of varieties.

Now, let us state our version of the Bondal--Orlov reconstruction. Here, $\spec^\ser \perf X$ is equipped with a ringed space structure, using only the triangulated category structure of $\perf X$.  
\begin{theorem}[Theorem \ref{thm:Bondal--Orlov}]\label{thm: intro1}
    Let $X$ be a Gorenstein proper variety of dimension $n$ over a field $k$ and suppose $\spec^\ser \perf X$ is a separated scheme over $k$ whose $n$-dimensional connected components (with reduced structure) are all isomorphic. Then, the following assertions hold:
    \begin{enumerate}
        \item The variety $X$ can be reconstructed as an $n$-dimensional connected component of $\spec^\ser \perf X$ with reduced structure, which is constructed by using only the triangulated category structure of $\perf X$. 
        \item If there is an exact equivalence $\perf X \simeq \perf Y$ for a variety $Y$, then $X \iso Y$. \qedhere
    \end{enumerate}
\end{theorem}
For example, we can check if $X$ is an elliptic curve over an algebraically closed field or if $X$ is a Gorenstein projective variety with (anti)ample canonical bundle, then $X$ satisfies the supposition of the theorem. Indeed, for the former, the Serre invariant locus is a disjoint union of the original elliptic curves and for the latter, the Serre invariant locus is $X$ itself. It is also worth mentioning that the Serre invariant locus $\spec^\ser \perf X$ may have components of different dimension, and this possibility is geometrically subtle (Example \ref{example: preoplarization in perf E}, \ref{example: pt-ample and fixed points}). 

To see our version further indeed generalizes the version of \cite{ito2024new}, let us introduce the following notion.
\begin{definition}
    For a scheme $X$, we say a line bundle $\ecal L$ on $X$ is \textbf{$\tens$-ample} if the smallest thick category containing all the tensor powers of $\ecal L$ is the whole category $\perf X$, i.e., 
    \[
    \bra{\ecal L^{\tens n} \mid n \in \bb Z} = \perf X. 
    \]
    For a line bundle $\ecal L$ on $X$, let $\tau_\ecal L$ denote the autoequivalence on $\perf X$ given by $- \tens_{\ecal O_X}^\bb L \ecal L$.
\end{definition}
\begin{example}[Example \ref{example: tensor ample}] \ 
\begin{enumerate}
    \item For a Gorenstein proper variety $X$ with $\tens$-ample canonical bundle, we have 
    \[
    \spec^\ser \perf X = \spec_{\tens_X} \perf X \iso X.
    \]
    In particular, a Gorenstein proper variety with $\tens$-ample canonical bundle satisfies the supposition of Theorem \ref{thm: intro1}. 
    \item More specifically, a blow-up $S$ of a smooth projective surface with very ample canonical bundle at a point has the $\tens$-ample canonical bundle $\omega_S$, but $\omega_S$ is neither ample nor anti-ample. \qedhere
\end{enumerate}
\end{example}
Now, by abstracting the proof of Theorem \ref{thm: intro1}, we can realize what is playing an essential role here is the information of a single autoequivalence, the Serre functor, and its commutativity with exact equivalences. Moreover, we just observed that the Balmer spectrum can be realized as the fixed locus of a suitable single autoequivalence, which leads us to the following notions. 
\begin{definition} Let $\cal T$ be an essentially small triangulated category. 
\begin{enumerate}
    \item A pair $(\cal T, \tau)$ of a triangulated category $\cal T$ and its autoequivalence $\tau$ is called a \textbf{polarized triangulated category} (or a \textbf{pt-category} in short). In this context, let us refer to $\tau$ as a \textbf{polarization} on $\cal T$.
    \item For a pt-category $(\cal T, \tau)$, define its \textbf{pt-spectrum} to be $\spec^\tau \cal T = \{\cal P \in \spec_\vartriangle \cal T \mid \tau(\cal P) = \cal P\}$, where the structure sheaf is defined in the same way as the Matsui spectrum. 
    \item An exact functor $F:(\cal T_1, \tau_1) \to (\cal T_2,\tau_2)$ between pt-categories is said to be a \textbf{pt-functor} if there exists a natural isomorphism $F\circ \tau_1 \iso \tau_2 \circ F$.  \qedhere
\end{enumerate}
\end{definition}
A similar notion of a polarized triangulated category is also introduced in \cite{kuznetsov2021serre}, which is a special case of our definition, so it should not cause confusions. Here are some benefits of having the formalism of pt-categories in the context of tensor triangular geometry. 
\begin{itemize}
    \item For a nice enough tensor triangulated category, including the case of $(\perf X, \tens_X^\bb L)$ for a noetherian scheme $X$, we can view its Balmer spectrum as a pt-spectrum for a certain (weak) polarization (cf. Proposition \ref{prop: tt as weak pt}). This realization suggests that the study of pt-structures could provide new insights into tensor triangulated geometry (cf. Remark \ref{moduli}).   
    \item The classification of tt-structures on a given triangulated category is not yet well understood. However, classifying possible {polarizations} is often more tractable, as the autoequivalence group of \( \perf X \) is well understood for many varieties \( X \). In particular, we explicitly compute pt-spectra for \( \mathbb{P}^1 \) and an elliptic curve \( E \) (Examples \ref{example: preoplarization in perf P1}, \ref{example: preoplarization in perf E}). 
    \item In noncommutative geometry, the perfect derived category \( \perf R \) of right modules over a smooth proper dg-algebra \( R \) does not generally admit a natural tt-structure. However, it is known that the Serre functor still exists (\cite{shklyarov2007serre}), which provides a {canonical polarization}. In more geometric settings, categories with a Serre functor often arise, such as in the study of Kuznetsov components (\cite{kuznetsov2021serre}), which usually do not admit a natural tt-structure. Moreover, if a triangulated category contains a {spherical object}, it induces a polarization via its spherical twist (\cite{SeiTho01}). These structures suggest a role of pt-categories in capturing geometric structures beyond tensor triangular geometry.
\end{itemize}

The framework of pt-categories also raises the following natural questions in the context of the reconstruction of varieties.

\begin{question}
    Let \( X \) be a Gorenstein proper variety. Is there any sufficient categorical condition on a polarization \( \tau \) so that \( \spec^\tau \perf X \cong X \), or more generally so that \( \spec^\tau \perf X \) is a \textbf{Fourier--Mukai partner} of \( X \), i.e., 
    \[
    \perf X \simeq \perf \spec^\tau \perf X?
    \]
    Note that given such conditions, we can reconstruct all the possible Fourier--Mukai partners of \( X \) purely from the triangulated category structure of \( \perf X \).
\end{question}

Now, as more concrete applications of our framework, we provide new proofs and generalizations of several works of Favero in \cite{FAVERO20121955}. One such result is the following, which is stated more generally as Theorem \ref{thm: favero}. 
\begin{theorem}\label{intro:thm3}
    Let $X$ and $Y$ be proper varieties and suppose there is a pt-equivalence
    \[
    \Phi: (\perf X, \tau_{\ecal L}) \overset{\sim}{\to} (\perf Y, \tau),
    \]
    where (i) $\tau$ and $\Phi$ are Fourier--Mukai transforms, (ii) $\ecal L$ is a line bundle on $X$ such that $\ecal L^{\tens m}$ is a $\tens$-ample line bundle and $\h^l(X,\ecal L^{\tens m})\neq 0$ for some $m, l\in \bb Z$, and (iii) $\tau$ is a standard autoequivalence, i.e., $$\tau = f_*\circ\tau_\ecal M[n] \in \aut(X) \ltimes \pic(X) \times \bb Z[1].$$
    Then, we have $f^m = \id_Y$ and $X\iso Y$. 
\end{theorem}

In the rest of the paper, we explore broader applications of the perspectives of pt-categories. 

In Subsection \ref{subsec: Iitaka}, we begin by observing that given a projective variety $X$ with an ample line bundle $\ecal L$, there are two ways of recovering $X$ from $(\perf X, \tau_\ecal L)$:
\[
\spec^{\tau_\ecal L}\perf X \iso X \iso \proj \cal R_\ecal L,
\]
where $\cal R_\ecal L$ denotes the \textbf{section ring} of $\ecal L$: $$\cal R_\ecal L:= \oplus_{n\geq 0}\Gamma(X, \ecal L^{\tens n}) = \oplus_{n\geq 0} \hom_{\perf X}(\ecal O_X, \tau_\ecal L^n(\ecal O_X)).$$ As a side remark, this highlights that the reconstruction via $\proj \cal R_\ecal L$ actually uses the extra information of the object $\ecal O_X \in \perf X$. Now, based on the isomorphism $X \iso \proj \cal R_\ecal L$, Artin-Zhang (\cite{artin1994noncommutative}) introduced the \textbf{noncommutative projective scheme} ${\proj}^\sf{nc} A$ associated to a right noetherian $\bb N$-graded algebra $A$. It consists of the triple $(\operatorname{\sf{qgr}} A, s_A, A)$, where:
\begin{itemize}
    \item The symbol $\operatorname{\sf{qgr}} A$ denotes an abelian category constructed from the abelian category of graded $A$-modules so that for the section ring $\cal R_\ecal L$ of an ample line bundle $\ecal L$ on $X$, we get an equivalence  $\operatorname{\sf{qgr}}\cal R_\ecal L\simeq \operatorname{\sf{coh}}X$, where $\operatorname{\sf{coh}}X$ is the abelian category of coherent sheaves on $X$. In particular, we can think of $\operatorname{\sf{qgr}}A$ as the abelian category of coherent sheaves on the noncommutative projective scheme $\proj^\sf{nc}A$.
    \item The symbol $s_A$ denotes an exact autoequivalence of $\operatorname{\sf{qgr}} A$ and we have $A \in \operatorname{\sf{qgr}} A$. Under the equivalence $\operatorname{\sf{qgr}}\cal R_\ecal L \simeq \operatorname{\sf{coh}}X$, $s_{\cal R_\ecal L}$ and $\cal R_\ecal L$ correspond to $-\tens \ecal L$ and $\ecal O_X$, respectively.
\end{itemize} 
 By suitably deriving $\operatorname{\sf{qgr}}A$ to $\perf \proj^\sf{nc}A$, we obtain the chain of isomorphisms:
\[
\spec^{\tau_\ecal L}\perf X  \iso \proj \cal R_\ecal L \iso \spec^{s_{\cal R_\ecal L}}\perf \proj^\sf{nc} \cal R_\ecal L.
\]

Now, we can make two key observations from the isomorphisms:
\begin{itemize}
    \item For a general right noetherian $\bb N$-graded algebra $A$, $\proj A$ no longer makes sense, but the ringed space $\spec^{s_{A}}\perf \proj^\sf{nc} A$ remains well-defined, so we propose it is a geometric realization of $\proj^\sf{nc} A$. A natural direction for future research is to investigate the extent to which the geometry of $\proj^\sf{nc} A$ is captured by $\spec^{s_{A}}\perf \proj^\sf{nc} A$. 
    \item When we set $\ecal L=\omega_X$ for a smooth projective variety $X$, the left hand side becomes the Serre invariant locus $\spec^\ser \perf X$ while the right hand side becomes the so called canonical model $X^\sf{can} = \proj \ecal R_{\omega_X}$ of $X$, which is a birational invariant of smooth projective varieties. When $\omega_X$ has positive Iitaka dimension, $X$ and $X^\sf{can}$ are related by a canonical rational map $X \ratmap X^\sf{can}$, called the \textbf{Iitaka fibration}. Namely, in such a case, we have the following relations:
    \[
    \spec^\ser \perf X \supset \spec_{\tens_X} \perf X \iso X \ratmap X^\sf{can} = \proj \cal R_{\omega_X} \iso \spec^{s_{\cal R_{\omega_X}}}\perf \proj^\sf{nc} \cal R_{\omega_X}.
    \]
    Since both the left-hand side and the right-hand side are birational invariants of Fourier--Mukai partners, further investigation can lead to a universal understanding and categorification of Iitaka fibrations across all such partners. 
\end{itemize}

In Subsection \ref{subsec: mirror}, we explore whether algebraic mirror partners in homological mirror symmetry can be realized as pt-spectra. To fix notations, suppose we have an equivalence  
\[
\perf X \simeq \operatorname{\mathsf{Fuk}}(M,\omega),
\]
where $X$ is a smooth variety and $\operatorname{\mathsf{Fuk}}(M,\omega)$ denotes a Fukaya-like category associated to a symplectic manifold $(M,\omega)$, possibly with additional data. Given that there can exist non-isomorphic Fourier--Mukai partners, the following natural questions arise.  

\begin{question}
    Given a symplectic manifold $(M,\omega)$, is there a canonical choice of an algebraic mirror partner? What data of the symplectic manifold determines its mirror partner(s)? 
\end{question}

From our perspective, these questions can be reformulated as follows:  

\begin{question}
    Is there a symplecto-geometric method for constructing a polarization $\tau$ on $\operatorname{\mathsf{Fuk}}(M, \omega)$ such that $\spec^\tau \operatorname{\mathsf{Fuk}}(M,\omega)$ defines an algebraic mirror partner of a symplectic manifold $(M,\omega)$ possibly together with additional data? Furthermore, how canonical is such a mirror partner?  
\end{question}

In certain cases, the answer is affirmative. In this paper, we focus on homological mirror symmetry of log Calabi-Yau surfaces, as studied in \cite{hacking2021symplectomorphisms} and \cite{hacking2023homological}. In particular, we reinterpret their results to show that algebraic mirrors can indeed be constructed as pt-spectra of Fukaya categories, using only symplecto-geometric data (Proposition \ref{prop: mirror partner as spectra}). Our approach makes the relation of choices of Lagrangian fibrations and Fourier--Mukai partners under homological mirror symmetry more explicit. 

Finally, in the appendix, we briefly discuss the functoriality of pt-spectra.  

The author would like to express gratitude for insightful discussions with Hiroki Matsui, David Nadler, John Nolan and Noah Olander.

\section{Geometry in Matsui spectrum}

Throughout this paper, we work over a field $k$. In particular, a scheme and an algebra are assumed to be over $k$ and a variety is assumed to be an integral and separated scheme of finite type over $k$. Given a ringed space $X$, $|X|$ denotes its underlying topological space. A module over a ring $R$ is a right $R$-modules unless otherwise specified and $\D(R)$ denotes the unbounded derived category of (right) $R$-modules and $\perf R$ denotes the perfect derived category of (right) $R$-modules. For a quiver $\sf Q$, let $k\sf Q$ denote its path algebra over $k$ and let $\tens_\sf{rep}$ denote the vertex-wise tensor product for (right) $k\sf Q$-modules (thought of as quiver representations over $\sf{Q}$). 

Now, let us quickly fix notations for Balmer's theory of tensor triangulated geometry and Matsui's theory of triangular spectra. For details of constructions and results, see \cites{Balmer_2005, matsui2023triangular}. 
\begin{notation}\label{notationa: prelim}  A triangulated category is assumed to be $k$-linear and \textbf{essentially small} (i.e., having a set of isomorphism classes of objects) and functors/structures are assumed to be $k$-linear.
\begin{enumerate} 
    \item Let $\mathcal{T}$ be a triangulated category and let $\Th(\mathcal{T})$ denote the set of \textbf{thick subcategories} (i.e., triangulated full subcategories closed under direct summand). Note that $\Th(\mathcal{T})$ is a poset (and indeed a complete and distributive lattice) with respect to inclusions and that we can equip $\Th(\mathcal{T})$ with topology, which we call the \textbf{Balmer topology}, whose closed subsets are of the form $$V(\mathcal{E}) = \{\mathcal{I} \in \Th(\mathcal{T})\mid \mathcal{
    I}\cap \mathcal{E} = \emptyset \},$$ where $\mathcal{E} \subset \ob \mathcal{T}$ is a collection of objects. Also, for a collection $\cal E$ of objects in $\cal T$, let $\bra{\cal E}$ denote the smallest thick subcategory containing $\cal E$. 
    \item We say a thick subcategory $\mathcal{P}$ of a triangulated category $\cal T$ is a \textbf{prime thick subcategory} if the subposet $\{\mathcal{I}\mid \mathcal{I}\supsetneq \mathcal{P}\} \subset \Th(\mathcal{T})$ has the smallest element. Let $\spc_\vartriangle \mathcal T$ denote the subspace of prime thick subcategories in $\Th\mathcal{T}$, which we call the \textbf{Matsui spectrum} or the \textbf{triangular spectrum} of $\cal T$. 
    \item Let $(\mathcal T,\otimes)$ be a \textbf{tensor triangulated category} (in short, a \textbf{tt-category}), i.e., a triangulated category $\cal T$ with a symmetric monoidal structure $\tens$ that is triangulated in each variable, called a \textbf{tt-structure}. A \textbf{tt-functor} is a monoidal functor between tt-categories whose underlying functor is triangulated.
    \item Let $\cal T$ be a triangulated category. A smooth projective variety $X$ is said to be a \textbf{Fourier--Mukai partner} of $\cal T$ if there exists a triangulated equivalence $\perf X \simeq \cal T$. Let $\FM \cal T$ denote the set of isomorphism classes of Fourier--Mukai partners of $\cal T$. 
    \item Let  $(\cal T,\tens)$ be a tt-category. A thick subcategory $\mathcal I$ is said to be a \textbf{$\otimes$-ideal} if for any $M \in \mathcal T$ and $N \in \mathcal I$, we have $M\otimes N \in \mathcal I$. Similarly, a $\tens$-ideal $\cal P$ is said to be a \textbf{prime $\otimes$-ideal} if $\cal P \neq \cal T$ and $M\tens N \in \cal P$ implies $M \in \cal P$ or $N \in \cal P$. Let $\spc_\otimes \mathcal T$ denote the subspace of prime $\otimes$-ideals in $\Th\mathcal{T}$, which we call the \textbf{Balmer spectrum} or the \textbf{tt-spectrum} of $(\cal T,\tens)$.
    \item Let $(\cal T, \tens)$ be a tt-category. We say an object $M \in \cal T$ is \textbf{$\tens$-invertible} if there exists $N \in \cal T$ such that $M \tens N \iso \bb 1$ and let $\pic(\cal T,\tens)$ denote the group of isomorphism classes of $\tens$-invertible objects. For an invertible object $M \in \pic (\cal T, \tens)$, let
    \[
    \tau_M := -\tens M
    \]
    denote the corresponding autoequivalence of $\cal T$ if doing so will not cause any confusion about which tt-structure we are using. 
    \item The group $\auteq \mathcal{T}$ of the natural isomorphism classes of autoequivalences of $\mathcal{T}$ acts on $\Th(\mathcal{T})$ both as lattice isomorphisms and as homeomorphisms by $$\auteq \cal T \times \Th \cal T \ni (\tau, \mathcal I) \mapsto \tau\inv (\mathcal I) \in \Th \cal T.$$ The action restricts to an action on $\spc_\vartriangle \mathcal{T}$. On the other hand, the action does not restrict to $\spc_\tens \cal T$ in general. Letting $\End \cal T$ denote the monoid of natural isomorphism classes of triangulated endofunctors on $\Th \cal T$, this action naturally extends to the monoid action
    \[
    \End \cal T \times \Th \cal T \ni (\tau,\cal I) \mapsto \tau\inv(\cal I) \in \Th \cal T. 
    \]
    \item For $\tau \in \End \mathcal{T}$, let $$\spc^\tau\mathcal{T} =\{\cal P \in \spc_\vartriangle \cal T \mid \cal P = \tau\inv(\cal P)\} \subset \spc_\vartriangle \cal T$$ denote the subspace of fixed points by the action of $\tau$. When $\cal T$ admits the Serre functor $\bb S$, we define the \textbf{Serre invariant locus} to be $$\spc^\ser \cal T:= \spc^\bb S \cal T.$$ 
    Similarly, for $\tau \in \End \cal T$, write 
    \[
    \Th^\tau \cal T = \{ \cal I \in \Th \cal T \mid \cal I = \tau\inv(\cal I) \}. \qedhere
    \]
\end{enumerate} 
\end{notation}
Let us recall the construction of the structure sheaf on the Balmer spectrum for readers' convenience. 
\begin{notation}\label{notation: models of loc and ic}
    In the rest of the paper, we fix the following models for localization and idempotent completion, which are the same as ones fixed in \cite{Balmer_2002}*{\href{https://arxiv.org/pdf/math/0111049.pdf}{2.11}, \href{https://arxiv.org/pdf/math/0111049.pdf}{2.12}}. See loc. cit. for more details. 
    \begin{enumerate}
        \item Let $\cal K$ be a triangulated category and let $\cal I$ be a thick subcategory of $\cal K$. Then, the localization $\cal K/\cal I$ can be modeled by setting objects to be the same as ones of $\cal K$ and morphisms to be equivalence classes of fractions $$\overset{s}\leftarrow \ {\rightarrow}$$ with $\operatorname{cone}(s) \in \cal I$ (cf. e.g., \cite{Verdier}*{\href{http://www.numdam.org/item/AST_1996__239__R1_0}{Chapter II}}).   
        Note that given a chain of thick subcategories $\cal I \subset \cal J \subset \cal K$, there is a canonical strictly commutative diagram
        \begin{center}
            \begin{tikzcd}
                         & \cal K \arrow[ld] \arrow[rd] &               \\
\cal K/\cal I \arrow[rr] &                              & \cal K/\cal J
\end{tikzcd}
        \end{center}
        where each functor is the identity on the object and the canonical quotient map on the hom set.
        \item Let $\cal K$ be a triangulated category (or more generally an additive category). We fix the model of idempotent completion $\cal K^\natural$ to be the \textbf{Karoubi envelope}, i.e., we set objects to be the pair $(M,e)$, where $e:M\to M$ is an idempotent in an object $M$, and morphisms $(M,e) \to (N,f)$ are morphisms $\phi:N\to N$ in $\cal K$ such that $f\circ \phi = \phi = \phi \circ e$. Note any triangulated functor $F: \cal K \to \cal L$ induces a unique canonical functor 
        \[
        F^\natural:\cal K^\natural \to \cal L^\natural,
        \]
        which sends an object $(M,e)$ to $(F(M),F(e))$ and a morphism $\phi$ to $F(\phi)$.
    \end{enumerate}
    For a triangulated category $\cal K$ and a thick subcategory $\cal I \subset \cal K$, we write the compositions of the canonical functors by  
    \[
    q_\cal I: \cal K \to \cal K/\cal I \to  (\cal K/\cal I)^\natural. \qedhere
    \]
\end{notation}
One reason for fixing models is to make sense of the strict equalities in the following construction:
\begin{construction}\label{construction:restriction maps}
    Let $\cal K$ be a triangulated category and let $\cal I \subset \cal J \subset \cal K$ be a chain of thick subcategories. Then, we get a canonical functor
    \[
    \rho_{\cal J \cal I}:(\cal K/\cal I)^\natural \to (\cal K/\cal J)^\natural 
    \]
    induced by the canonical functor $\cal K/\cal I \to \cal K/\cal J$. In particular, the following diagram strictly commutes:
    \begin{center}
        \DisableQuotes
        \begin{tikzcd}
                                                            & \cal K \arrow[ld] \arrow[rd] &                          \\
\cal K/\cal I \arrow[rr] \arrow[d]                          &                              & \cal K/\cal J \arrow[d]  \\
(\cal K/\cal I)^\natural \arrow[rr, "\rho_{\cal J\cal I}"'] &                              & (\cal K/\cal J)^\natural
\end{tikzcd}
    \end{center}
    and hence we have the strict equality
    \[
    q_{\cal J} = \rho_{\cal J\cal I} \circ q_\cal I. 
    \]
Now, by construction, for any chain $\cal I \subset \cal J \subset \cal L \subset \cal K$ of thick subcategories, we have the strict equality:
\[
\rho_{\cal L \cal I} = \rho_{\cal L\cal J}\circ \rho_{\cal J\cal I}. \qedhere
\]
\end{construction}
Now, we construct a structure sheaf on the Balmer spectrum. 
\begin{construction}\label{construction: structure sheaf of tt-spectra}
    Let $(\cal T, \tens)$ be a tt-category with unit $\bb 1$. We define a structure sheaf $\ecal O_{\cal T, \tens}$ on $\spc_\tens \cal T$ as follows. First, for an open subset $U \subset \spc_\tens \cal T$, set $\cal T^U = \cap_{\cal P \in U} \cal P$, which is clearly a $\tens$-ideal (where we set $\cal T^\emp = \cal T$). Then, the triangulated category
    \[
    \cal T(U):= (\cal T/\cal T^U)^\natural 
    \]
    has a canonical tt-structure $\tens_U$ induced by $\tens$ whose unit $\bb 1_U$ is the image of $\bb 1$ under the canonical functor $q_U:=q_{\cal T^U}:\cal T \to \cal T(U)$ (cf. Construction \ref{construction:restriction maps}). Now, we set 
    \[
    \ecal O_{\cal T, \tens}^\sf{pre}(U) := \End_{\cal T(U)}(\bb 1_U),
    \]
    which is commutative since $\bb 1_U$ is the unit of the symmetric monoidal structure $\tens_U$. Now, recall the following claim, which follows by Construction \ref{construction:restriction maps}. 
    \begin{enmlem}[\cite{Balmer_2002}*{\href{https://arxiv.org/pdf/math/0111049.pdf}{Proposition 5.3}}]\label{lemma: balmer structure}
        Let $(\cal T,\tens)$ be a tt-category. Let $V_1 \subset V_2$ be open subsets of $\spc_{\tens}\cal T$. Then, for the canonical tt-functor $$\rho_{V_1,V_2}:=\rho_{\cal T^{V_1},\cal T^{V_2}}:\cal T(V_2) \to \cal T(V_1),$$ we have the strict equality $q_{V_1} = \rho_{V_1V_2}\circ q_{V_2}$. Moreover, for any triple of open subsets $V_1\subset V_2 \subset V_3$ in $\spc_{\tens}\cal T$, we have the strict equality:
        \[
        \rho_{V_1V_3} = \rho_{V_1V_2}\circ\rho_{V_2V_3}.\qedhere
        \]
    \end{enmlem}
    Therefore, the association
    \[
    \ecal O_{\cal T,\tens}^\sf{pre}:\l(U \subset V\r) \mapsto \l(\End_{\cal T(V)}(\bb 1_V) \overset{\rho_{UV}}{\longrightarrow}\End_{\cal T(U)}(\bb 1_U)\r)
    \]
    is a presheaf of commutative rings and we define the structure sheaf $\ecal O_{\cal T, \tens}$ on $\spc_\tens \cal T$ to be its sheafification. The corresponding ringed space is denoted by $\spec_\tens \cal T$. 

    Note that we can generalize the construction to any subspace $X \subset \Th \cal T$ and an object $M \in \cal T$ instead of $\spc_\tens \cal T \subset \spc_\vartriangle \cal T$ and $\bb 1 \in \cal T$. More precisely, the association
    \[
    \ecal O_{X,M}^\sf{pre}:\l(U \subset V \subset X\r) \mapsto \l(\End_{\cal T(V)}(q_V(M)) \overset{\rho_{UV}}{\longrightarrow}\End_{\cal T(U)}(q_U(M))\r)
    \]
    is a presheaf of (not necessarily commutative) rings, where $\rho_{UV}$'s and $q_U$'s are suitably generalized, e.g., for any open subset $U \subset X$, we set 
\[
\cal T(U):= (\cal T/\cap_{\cal P \in U} \cal P)^\natural.
\] Now, we can define the sheaf $\ecal O_{X, M}$ of rings on $\spc_\tens \cal T$ to be its sheafification. 
    \end{construction}
We can also equip the Matsui spectrum with a ringed space structure $\spec_\vartriangle \cal T = (\spc_\vartriangle \cal T, \ecal O_{\cal T, \vartriangle})$ as follows:
\begin{construction}\label{constr: Matsui structure sheaf}
Let $\cal T$ be a triangulated category. The \textbf{center} $Z(\cal T)$ of a triangulated category $\cal T$ is defined to be the commutative ring of natural transformations $\alpha: \id_\cal T \to \id_\cal T$ with $\alpha[1] = [1] \alpha$, where the multiplication is defined by composition. We say an element $\alpha \in Z(\cal T)$ is \textbf{locally nilpotent} if $\alpha_M$ is a nilpotent element in the endomorphism ring $\End_\cal T(M)$ for each $M \in \cal T$. Let $Z(\cal T)_\sf{lnil}$ denote the ideal of locally nilpotent elements and set $Z(\cal T)_\sf{lrad}:= Z(\cal T)/Z(\cal T)_\sf{lnil}$.  

As in Construction \ref{construction: structure sheaf of tt-spectra}, we can make the association 
\[
U \mapsto Z(\cal T(U))_\sf{lred}. 
\]
into a presheaf $\ecal O_{\cal T,\vartriangle}^\sf {pre}$ on $\spc_\vartriangle \cal T$ and define the structure sheaf $\ecal O_{\cal T,\vartriangle}$ on $\spc_\vartriangle \cal T$ to be its sheafification. The corresponding ringed space is denoted by $\spec_\vartriangle\cal T$. See \cite{matsui2023triangular} for more details. 

Furthermore, as remarked in \cite{hirano2024FMlocusK3}*{\href{https://arxiv.org/pdf/2405.01169}{Section 1.3}} (see also \cite{Balmer_2002}*{\href{https://www.math.ucla.edu/~balmer/Pubfile/Reconstr.pdf}{Definition 6.6.}}), note that the same construction indeed defines a sheaf on any subspace $X \subset \Th \cal T$, which will be denoted by $\ecal O_{X, \vartriangle}$ or simply $\ecal O_X$ if doing so may not cause any confusion. To be more precise, $\ecal O_{X,\vartriangle}$ is the sheafification of the presheaf
\[
X \underset{\text{open}}{\supset} U \mapsto Z(\cal T(U))_\sf{lred},
\]
where as before we set 
\[
\cal T(U):= (\cal T/\cap_{\cal P \in U} \cal P)^\natural.
\]
Note Balmer's definition (\cite{Balmer_2002}*{\href{https://www.math.ucla.edu/~balmer/Pubfile/Reconstr.pdf}{Definition 6.6.}}) corresponds to the case when we set $X = \spc_\tens \cal T$ for a tt-structure $\tens$ on $\cal T$. Namely, using their notation, we have the relation
\[
\operatorname{Space}_\text{pt.red}(\cal T,\tens) = (\spc_\tens \cal T, \ecal O_{\spc_\tens \cal T,\vartriangle}). \qedhere
\]
\end{construction}
Now, the following are important aspects of the theories of Balmer and Matsui spectra in terms of algebraic geometry, where the order of references aligns with the order of the statements.
\begin{theorem}[\cite{Balmer_2005}*{\href{https://arxiv.org/pdf/math/0409360.pdf}{Theorem 6.3}}, \cite{Matsui_2021}*{\href{https://arxiv.org/pdf/2102.11317.pdf}{Corollary 3.8}}, \cite{Balmer_2002}*{\href{https://www.math.ucla.edu/~balmer/Pubfile/Reconstr.pdf}{Theorem 8.5.}}, \cite{matsukawa2025spectrum}*{\href{https://arxiv.org/abs/2505.02724v1}{Theorem 2.11.}}]\label{tt and triangle}
    Let $\cal T$ be a triangulated category with $\cal T \simeq \perf X$ for a noetherian scheme $X$. Fix a tt-structure $\tens$ on $\cal T$ so that $(\cal T, \tens)$ is tt-equivalent to $(\perf X, \tens_{\ecal O_X}^\bb L)$. Then, the following assertions hold:
    \begin{enumerate}
        \item We have $\spec_{\tens} \cal T \iso X$ as ringed spaces. 
        \item We have the inclusion $\spc_{\tens} \cal T \subset \spc_\vartriangle \cal T$ as topological spaces. 
        \item We have $(\spc_\tens \cal T, \ecal O_{\spc_\tens \cal T, \vartriangle})\iso X_\sf{red}$ as ringed spaces (cf. Construction \ref{constr: Matsui structure sheaf} for notations). 
        \item We have an open immersion $$X_\sf{red} \iso (\spc_\tens \cal T, \ecal O_{\spc_\tens \cal T, \vartriangle})\inj \spec_\vartriangle \cal T$$ of ringed spaces whose underlying continuous map is the inclusion in (ii). \qedhere
    \end{enumerate}
\end{theorem}
\begin{proof}
    Only for readers' convenience, we reproduce a proof of part (iv) although it is a mere specialization of the proof of Matsukawa \cite{matsukawa2025spectrum} to our setting. Let $X = \bigcup_{i \in I}U_i$ be an affine open cover and set $Z_i=X\setminus U_i$. Also, fix a split generator $\ecal E$ of $\perf X$.  By \cite{thomason1997classification}*{Lemma 3.4.} there is $\ecal F_i \in \perf X$ such that $\supp \ecal F_i = Z_i$. Since $\bra{\ecal F_i \tens_X^\bb L \ecal E} = \bra{\ecal F_i \tens_{X}^\bb L\perf (X)}\subset \perf X$ is a $\tens$-ideal and $$\ecal F_i \in \bra{\ecal F_i \tens_X^\bb L \ecal E}\subset \perf_{Z_i} X := \{\ecal G \in \perf X\mid \supp \ecal G\subset Z_i\}$$ we have $\bra{\ecal F_i\tens_X^\bb L \ecal E} = \perf _{Z_i} X$ by \cite{thomason1997classification}*{Theorem 3.15.}. Now, by \cite{matsui2023triangular}*{\href{https://arxiv.org/pdf/2301.03168}{Corollary 4.4.}}, the natural inclusion
    \[
    \spc_\vartriangle\perf X/\perf_{Z_i}X \inj \spc_\vartriangle\perf X
    \]
    is an open embedding, where we identify $\spc_\vartriangle\perf X/\perf_{Z_i}X$ with the subspace on prime thick subcategories of $\perf X$ containing $\perf_{Z_i}X$. Now, recall from \cite{Balmer_2002} that
    \[
    \spc_{\tens_X^\bb L}\perf X= \{S_X(x) \mid x \in X\},
    \]
    where $S_X(x) = \{\ecal F \in \perf X \mid F_x \iso 0 \text{ in $\perf \ecal O_{X,x}$}\}$ (and $x \in X$ is not necessarily a closed point). We claim 
    \[
    \spc_\vartriangle\perf X/\perf_{Z_i}X = \{S_X(x)\mid x \in U_i\}. 
    \]
    Indeed, it is clear $\perf_{Z_i}X \subset S_X(x)$ for $x \in U_i$, so we have $\{S_X(x)\mid x \in U_i\} \subset \spc_\vartriangle\perf X/\perf_{Z_i}X$. On the other hand, recall that an open immersion $j_i:U_i \inj X$ induces a fully faithful dense functor
    \[
    j_i^*: \perf X/\perf_{Z_i}X \inj \perf U_i 
    \]
    (cf. \cite{thomason2007higher} and \cite{Balmer_2002}*{\href{https://www.math.ucla.edu/~balmer/Pubfile/Reconstr.pdf}{Theorem 2.13.}}) and since $j_i^*S_X(x) = S_{U_i}(x)$ for $x \in U_i$, we have a commutative diagram:
    \begin{center}
        \DisableQuotes
        \begin{tikzcd}
\{S_X(x)\mid x \in U_i\} \arrow[d, "j_i^*"'] \arrow[r, hook] & \spc_\vartriangle\perf X/\perf_{Z_i}X \arrow[d, "j_i^*"] \\
\spc_{\tens_{U_i}^\bb L}\perf U_i \arrow[r, hook]            & \spc_\vartriangle\perf U _i                               
\end{tikzcd}
    \end{center}
    where vertical arrows are bijections by \cite{Matsui_2021}*{\href{https://arxiv.org/pdf/2102.11317}{Proposition 2.11.}}. Moreover, the bottom inclusion is the identity by \cite{Matsui_2021}*{\href{https://arxiv.org/pdf/2102.11317}{Corollary 4.9.}} (as any thick subcategory of $\perf U_i$ is a $\tens$-ideal) and therefore we see that the top inclusion is the identity as desired. Now, we have that
    \[
    \spc_{\tens_X^\bb L}\perf X = \bigcup_{i \in I} \spc_\vartriangle\perf X/\perf_{Z_i}X \subset \spc_\vartriangle \perf X
    \]
    is an open subspace and hence we are done by \cite{matsui2023triangular}*{\href{https://arxiv.org/pdf/2301.03168}{Theorem 4.6.}}. 
\end{proof}
Now, using the same notations as in Theorem \ref{tt and triangle}, we can interpret those results by the following slogans:
    \begin{itemize}
        \item Part (i) claims that a noetherian scheme can be viewed as a choice of a tt-structure on its perfect derived category.
        \item Part (ii) claims that topologically such a choice of a tt-structure may be viewed as a choice of a subspace in the Matsui spectrum.
        \item Part (iii) claims that the reduced structure sheaf on $|X|$ can be recovered as the structure sheaf $\ecal O_{\spc_\otimes \cal T, \vartriangle}$ of $\spc_\otimes \cal T$ in the sense of Construction \ref{constr: Matsui structure sheaf}, which only uses the triangulated category structure once we identify the subspace $\spc_\tens \cal T \subset \spc_\vartriangle \cal T$. 
        \item Part (iv) claims that the structure sheaf of $X_\sf{red}$ can be indeed recovered as the restriction of the structure sheaf $\ecal O_{\cal T, \vartriangle}$ of the entire Matsui spectrum $\spc_\vartriangle \cal T$. 
    \end{itemize}
An upshot is that understanding subspaces of the Matsui spectra provides a novel way to understand relations of derived categories and algebraic geometry. From the next section, we observe that the fixed points of an autoequivalence (or more generally an endofunctor) provides a way to specify such subspaces.

To finish our preparation, let us collect some technical results that refine the relation between prime $\tens$-ideals and prime thick subcategories for later use. 
\begin{theorem}[\cite{Matsui_2021}*{\href{https://arxiv.org/pdf/2102.11317.pdf}{Proposition 4.8}}, \cite{matsui2023triangular}*{\href{https://arxiv.org/pdf/2301.03168.pdf}{Corollary 3.2}, \href{https://arxiv.org/pdf/2301.03168.pdf}{Theorem 4.6}}]\label{thm: matsui prime thick ideal}
    Let $(\cal T, \tens)$ be a rigid tt-category. Recall by \cite{Balmer_2005}*{\href{https://www.math.ucla.edu/~balmer/Pubfile/Spectrum.pdf}{Proposition 4.4.}}, any $\tens$-ideal $\cal I$ in a rigid tt-category is radical (in the sense that if $X^{\tens n} \in \cal I$ for some $n\geq 1$, then $X \in \cal I$).
    \begin{enumerate}
        \item Any prime thick subcategory that is a $\tens$-ideal is a prime $\tens$-ideal.
        \item If $(\cal T,\tens)$ is moreover idempotent complete and locally monogenic (cf. \cite{matsui2023triangular}*{p.9} for definition) with noetherian Balmer spectrum, then any prime $\tens$-ideal is a prime thick subcategory, i.e., we have $\spc_\tens \cal T\subset \spc_\vartriangle \cal T$.
        \item For a noetherian scheme $X$, the tt-category $(\perf X,\tens_{\ecal O_X}^\bb L)$ is rigid, idempotent complete, and locally monogenic. In particular, a thick subcategory of $\perf X$ is a prime $\tens$-ideal if and only if it is a prime thick subcategory that is a $\tens$-ideal (which is also directly shown in \cite{Matsui_2021}*{\href{https://arxiv.org/pdf/2102.11317.pdf}{Corollary 4.9}}). \qedhere
    \end{enumerate}    
\end{theorem}
For convenience, let us use the following notion in this paper. 
\begin{definition}
    A tt-category $(\cal T,\tens)$ (or a tt-structure $\tens$) is said to be \textbf{$M$-compatible} ($M$ for Matsui) if we have that a thick subcategory of $\cal T$ is a prime $\tens$-ideal if and only it is a prime thick subcategory that is a $\tens$-ideal.  
\end{definition}
By Theorem \ref{thm: matsui prime thick ideal}, a rigid, idempotent complete and locally monogenic tt-category with noetherian Balmer spectrum is $M$-compatible. On the other hand, this is not a necessary condition.
\begin{example}\label{example: non rigid}
    For a finite acyclic quiver $\sf Q$, the tt-category $(\perf k\sf Q, \tens_{\sf rep}^\bb L)$ is in general not rigid, but since $\spec_{\tens_\sf Q}\perf k\sf Q$ consists of maximal thick subcategories (cf. \cite{SirLiu13}), we have $$\spec_{\tens_\sf{rep}^\bb L}\perf k\sf Q \subset \spec_\vartriangle \perf k\sf Q.$$ Moreover, it is easy to see that a thick subcategory is a prime $\tens$-ideal if and only if it is a prime thick subcategory that is a $\tens$-ideal. 
\end{example}

\section{Fixed points of autoequivalences and polarizations}
In this section, we introduce notions of polarizations on triangulated categories and their corresponding spectra in relation with tensor triangulated geometry although studies of polarizations turn out to be rich in its own right. To mimic the construction of tt-spectra, let us introduce our guiding example from algebraic geometry, which was essentially observed in \cites{HO22, matsui2023triangular} through some specific cases and made explicit in \cite{ito2024new}. First, let us recall the following basic tools:
\begin{definition}\label{def: split generator}
    Let $\cal T$ be a triangulated category. We say $E \in \cal T$ is a \textbf{split generator} if any object in $\cal T$ is the direct summand of an object obtained by taking iterative extensions of finite direct sums of shifts of $E$, which is equivalent to saying that the smallest thick subcategory containing $E$ is $\cal T$, i.e., we have $\bra{E} = \cal T$ (e.g., cf. \cite{stacks-project}*{\href{https://stacks.math.columbia.edu/tag/0ATG}{Tag 0ATG}}). Recall if $X$ is a quasi-compact and quasi-separated scheme, then $\perf X$ has a split generator (\cite{bondal2002generators}*{Corollary 3.1.2.}). 
\end{definition}
\begin{lemma}\label{lem: split gen}
    Let $(\cal T,\tens)$ be a tt-category with a split generator $E$. Then, a thick subcategory $\cal I \subset \cal T$ is a $\tens$-ideal if and only if $\cal I \tens E \subset \cal I$.
\end{lemma}
\begin{proof}
    One direction is clear. Conversely, suppose $\cal I \tens E \subset \cal I$ for $\cal I \in \Th \cal T$. Let $\cal J = \{M \in \cal T\mid \cal I \tens M \subset \cal I\}$. Then $\cal J$ is clearly a thick subcategory of $\cal T$. Now, since $E \in \cal J$, we have $\cal J = \cal T$, i.e., $\cal I$ is a $\tens$-ideal. 
\end{proof}
Now, the following is our guiding example. 
\begin{prop}\label{example:guiding}
    Let $X$ be a quasi-projective scheme (over any field) with an (anti-)ample line bundle $\ecal L$ on $X$ and write $\tens := \tens_{\ecal O_X}^\bb L$. Recalling Notation \ref{notationa: prelim} (vi), (viii), we have
\[
\spc_\tens \perf X = \spc^{\tau_\ecal L}\perf X \subset \spc_\vartriangle \perf X. \qedhere
\]
\end{prop}
\begin{proof}
    Although this is essentially \cite{ito2024new}*{\href{https://arxiv.org/pdf/2405.16776}{Lemma 3.1}}, we sketch the proof for readers' convenience. First, by \cite{Orlov_dimension}*{\href{https://arxiv.org/pdf/0804.1163.pdf}{Theorem 4}}, we see that there exists a finite subset $S \subset \bb Z$ such that $\oplus_{n\in S}\ecal L^{\tens n}$ is a split generator of $\perf X$, where we write $\ecal L^{\tens 0}:=\ecal O_X$. Now, note that for a thick subcategory $\cal I \subset \perf X$, if we have $\cal I \tens \ecal L = \cal I$, then we have $\cal I \tens \l(\oplus_{n\in S}\ecal L^{\tens n} \r)\subset \cal I$. Therefore, by Lemma \ref{lem: split gen} and Theorem \ref{thm: matsui prime thick ideal} (iii), we have 
    \[
\spc_\tens \perf X = \spc^{\tau_\ecal L}\perf X \subset \spc_\vartriangle \perf X
\]
since the left hand side is the subspace of prime $\tens$-ideals and the right hand side is the subspace of prime thick subcategories that are $\tens$-ideals. 
\end{proof}
Let us list some takeaways from this observation.
\begin{itemize}
    \item To get the space of fixed points, we only need a single autoequivalence $\tau_\ecal L \in \auteq \perf X$, which does not require any other extra structures on the triangulated category $\perf X$ such as monoidal structures.
    \item When $X$ is reduced, we can reconstruct the structure sheaf of the Balmer spectrum on $\spc_\tens \perf X = \spc^{\tau_ \ecal L}\perf X$ as the restriction of the structure sheaf of the Matsui spectrum by Theorem \ref{tt and triangle} (iv). Namely, we do not need the data corresponding to the tensor unit $\ecal O_X$ to define the structure sheaf (cf. Construction \ref{construction: structure sheaf of tt-spectra}). 
    \item It is also worth emphasizing classifications of tt-structures are not yet well-formulated while classifications of autoequivalences just happen inside of $\auteq \perf X$, which is well-understood in various cases. In other words, we can hope for classifying geometry of
    \[
    \spc^{\tau} \perf X \subset \spc_\vartriangle \perf X
    \]
    when we vary $\tau \in \auteq \perf X$. 
\end{itemize}
In the rest of the section, let us first generalize this phenomena with presence of tt-structures to clarify which classes of tt-categories are in our scope and later we will recast this phenomena in the absence of tt-structures so that it is applicable to more general classes of triangulated categories not admitting natural (symmetric) monoidal structures. 

\subsection{Generalization with tt-structures and $\tens$-ample objects}
The following property extracts the essence of the proof of Proposition \ref{example:guiding}.
\begin{definition} Let $(\cal T, \tens)$ be a tt-category. An object $\ecal L \in \cal T$ is said to be \textbf{$\tens$-ample} if $\ecal L$ is $\tens$-invertible and if $\bra{\ecal L^{\tens n} \mid n \in \bb Z} = \cal T$. By slight abuse of notations, for a scheme $X$, a line bundle $\ecal L$ on $X$ is said to be \textbf{$\tens$-ample (on $X$)} if $\ecal L$ is $\tens$-ample in $(\perf X, \tens_{\ecal O_X}^\bb L)$.
\end{definition}
Indeed, by the almost identical arguments as in the proof of Proposition \ref{example:guiding}, we immediately obtain the following generalization. 
\begin{prop}\label{prop: main autoequiv} 
    Let $(\cal T,\tens)$ be an $M$-compatible tt-category with a $\tens$-ample object $\ecal L$. Then, we have
    \[
    \spc_\tens \cal T = \spc^{\tau_ \ecal L} \cal T \subset \spc_\vartriangle \cal T. 
    \]
    In particular, if $\ecal L$ is a $\tens$-ample line bundle on a noetherian scheme $X$, then we have
    \[
    \spc_{\tens_X^\bb L} \perf X = \spc^{\tau_\ecal L}\perf X \underset{\sf{open}}{\subset} \spc_\vartriangle \perf X.\qedhere
    \]
\end{prop} 
\begin{proof}
    The containment $\spc_\tens \cal T \subset \spc^{\tau_\ecal L}\cal T$ is clear. Conversely, take $\cal P \in \spc^{\tau_\ecal L}\cal T$ and consider a thick subcategory $\cal I:= \{M \in \cal T \mid \cal P \tens M \subset \cal P\} \subset \cal T$. Then, since $\cal P \tens \ecal L^{\tens n} \subset \cal P$ for any $n \in \bb Z$, we have $\ecal L^{\tens n} \in \cal I$ for any $n \in \bb Z$ and therefore $\cal I = \cal T$ as $\ecal L$ is $\tens$-ample. Thus, $\cal P$ is a prime thick subcategory that is a $\tens$-ideal, so it is a prime $\tens$-ideal by Theorem \ref{thm: matsui prime thick ideal}.
\end{proof}
Now, even in algebro-geometric situations, Proposition \ref{prop: main autoequiv} is a strict generalization of Proposition \ref{example:guiding}. In fact, it is known to experts that not all $\tens$-ample line bundles are (anti-)ample. For completeness, let us provide examples, which the author essentially learned from \cite{259385} although our proof slightly differs. 

\begin{prop}\label{lem: tens-ample but not (anti)ample, blow up}
            Let $Y$ be a smooth projective surface over an algebraically closed field and let $\pi:X \to Y$ be the blow-up at a point $p \in Y$. Write $\tens:=\tens_{\ecal O_X}^\bb L$. Let $\ecal O_Y(1)$ be a very ample line bundle on $Y$ and let $E \subset X$ denote the exceptional divisor. Then, we have that
            \[
            \ecal L:= \ecal O_X(E) \tens \pi^* \ecal O_Y(1)
            \]
            is neither ample nor anti-ample, but is $\tens$-ample. 
        \end{prop}
        \begin{proof}
            First of all, $\ecal L$ is neither ample nor anti-ample (not even nef nor anti-nef) since $c_1(\ecal L)\cdot E = -1$ while $c_1(\ecal L)\cdot [C] > 0$ for any curve $C$ with $C\cap E = \emp$.

            Now, let us show $\ecal L$ is $\tens$-ample. Let $\cal C$ denote $\bra{\ecal L^{\tens n}\mid n \in \bb Z} \subset \perf X$ and we want to show every tensor power of the ample line bundle $\ecal M:= \ecal O_X(-E) \tens \pi^* \ecal O_Y(2)$ (e.g., cf. \cite{Har77}*{Exercise V.3.3.}) is contained in $\cal C$. First, take a closed point $x \in X\setminus E$. Then, there exists a hyperplane section $H \subset Y$ with $\pi(x) \in H$ and $p \not \in H$, which gives a section of $\pi^* \ecal O_Y(1)$ that precisely vanishes over $\pi\inv(H)$. Therefore, by tensoring with a section of $\ecal O_X(E)$ corresponding to a short exact sequence $0 \to \ecal O_X(-E) \to \ecal O_X \to \ecal O_E \to 0$, we get a section of $\ecal L$ that precisely vanishes over $\pi\inv(H) \cup E$, which gives a short exact sequence
            \[
            0 \to \ecal L^{\tens(-1)} \to \ecal O_X \to \ecal O_{\pi\inv(H)}\oplus \ecal O_E \to 0.
            \]
            Note that this implies $$\ecal O_{\pi \inv(H)}, \ecal O_E \in \cal C.$$ Furthermore, for each $n \in \bb Z$, by tensoring $\ecal L^{\tens n}$ to the short exact sequence above, we see that
            \[
            \ecal O_{\pi\inv(H)} \tens \ecal L^{\tens n} =  \ecal O_{\pi\inv(H)} \tens \pi^* \ecal O_Y(n) \in \cal C \quad \text{and} \quad \ecal O_E \tens \ecal L^{\tens n} = \ecal O_E(nE) \in \cal C. 
            \]
            Therefore, for each $n\in \bb Z$ by tensoring $\ecal O_X(nE)$ to the short exact sequence $0 \to \ecal O_X(-E) \to \ecal O_X \to \ecal O_E \to 0$, we inductively see that
            \[
            \ecal O_X(nE) \in \cal C. 
            \]
            Now, for each $n \in \bb Z_{>0}$ consider a short exact sequence
            \[
            0 \to \ecal O_X(-3nE) \to \ecal O_X \to \ecal Q_n \to 0
            \]
            so that $\ecal Q_n \in \cal C$. Note by tensoring with $\ecal O_X(2nE)$, we see that $$\ecal Q_n\tens \ecal O_X(2nE) \in \cal C.$$ By further tensoring with $\pi^*\ecal O_Y(2n)$, we obtain the short exact sequence
            \[
            0 \to \ecal M^{\tens n} \to \ecal L^{\tens 2n} \to \ecal Q_n\tens \ecal O_X(2nE)\tens \pi^*\ecal O_Y(n) \to 0.
            \]
            Now, since $\supp \ecal Q_n = E$, we have 
            \[
            \ecal Q_n\tens \ecal O_X(2nE)\tens \pi^*\ecal O_Y(n) = \ecal Q_n\tens \ecal O_X(2nE) \in \cal C,
            \]
            so $\ecal M^{\tens n} \in \cal C$ for any $n \in \bb Z_{>0}$ as desired. 
        \end{proof}
As non-geometric examples, let us consider the following simple case. 
\begin{example} For a finite acyclic quiver $\sf Q$, since we have $$\pic(\perf k\sf Q,\tens_\sf{rep}^\bb L) = \{\bb 1\},$$ the unit $\bb 1$ is the only $\tens$-invertible object. Also, note $\bb 1$ is $\tens$-ample if and only if $\bb 1$ split-generates $\perf \sf Q$. Thus, when $\sf Q$ has more than one vertex, $(\perf k\sf Q,\tens_{\sf{rep}}^\bb L)$ does not have any $\tens$-ample object.  
\end{example}

\subsection{Generalization in the absence of tt-structures and polarizations}
In the previous subsection, the data of an autoequivalence coming from tensoring with $\tens$-ample objects played an important role in recovering tt-spectra. Based on this observation, let us consider an analogue of tt-spectra in the absence of tt-structures. 
\begin{definition}\label{definition: ppt-category}
    Let $\cal T$ be a triangulated category. 
    \begin{enumerate}   
        \item For an autoequivalence $\tau \in \auteq \cal T$, the pair $(\cal T, \tau)$ is said to be a \textbf{polarized triangulated category} (in short, a \textbf{pt-category}). In this context, an autoequivalence of $\cal T$ may be interchangeably referred to as a \textbf{polarization} on $\cal T$. For a pt-category $(\cal T,\tau)$, define its \textbf{pt-spectrum} to be the ringed space
        \[
        \spec^\tau \cal T = (\spc^\tau \cal T, \ecal O_{\spc^\tau \cal T,\vartriangle}). 
        \]
        If there is no confusion, let us write $\ecal O_\tau:= \ecal O_{\spc^\tau \cal T,\vartriangle}$. 
        \item Let $(\cal T_1, \tau_1)$ and $(\cal T_2, \tau_2)$ be pt-categories. A \textbf{pt-functor} (resp. \textbf{pt-equivalence}) is defined to be a triangulated functor (resp. {equivalence}) $F: \cal T_1 \to \cal T_2$ such that there exists a natural isomorphism $\tau_2\circ F \iso F \circ \tau_1$.
    \end{enumerate}
    As in the case of Matsui spectra (cf. \cite{matsui2023triangular}*{\href{https://arxiv.org/pdf/2301.03168}{Proposition 4.2}}), a pt-equivalence $F:(\cal T_1,\tau_1) \overset{\sim}{\to}(\cal T_2,\tau_2)$ induces an isomorphism of pt-spectra as ringed spaces by $F\inv:\spec^{\tau_2} \cal T_2 \ni \cal P \mapsto F\inv(\cal P) \in \spec^{\tau_1} \cal T_1$. 
\end{definition} 
\begin{remark}\label{rem: not functorial}
    One immediate question is whether or not the association of pt-categories with pt-functors to pt-spectra as above is functorial. Unfortunately, it is not even well-defined on morphisms for the similar reason that the maximal ideals of a ring are not necessarily pulled back to maximal ideals by ring homomorphisms. Nevertheless, there is a way to construct a functorial theory, which we will discuss in Appendix \ref{append}.
\end{remark}
Generalizations to $\tens$-ample objects in the previous subsection revealed that there is subtlety among (anti)ample line bundles and $\tens$-ample line bundles, but further generalizations to pt-categories reveal there are even more subtleties for such properties of line bundles. We begin with a simpler example of $\bb P^1$ and then move on to the case of elliptic curves where we can already observe such complexity, namely, there is a line bundle that is not $\tens$-ample, but whose pt-spectrum still agrees with tt-spectrum at all $k$-points.
\begin{example}\label{example: preoplarization in perf P1}
     Let $X$ be a smooth projective variety with ample (anti-)canonical bundle. Then, by \cite{bondal_orlov_2001}, polarizations on $\perf X$ are classified by 
\[
\auteq \perf X = \aut X \ltimes \pic X \times \bb Z[1]. 
\]
In particular, if $X = \bb P^1$, then $\auteq \perf \bb P^1$ is a group generated by $\aut \bb P^1 \iso \operatorname{PGL}(2)$, $\pic(X) = \bb Z$ and shifts $\bb Z[1]$. Also, note that by \cite{Matsui_2021}*{\href{https://arxiv.org/pdf/2102.11317.pdf}{Example 4.10}}, we have
\[
\spc_\vartriangle \perf \bb P^1 = \bb P^1 \sqcup \bb Z,
\]
where $\bb Z$ is equipped with discrete topology and each point corresponds to the thick subcategory generated by $\ecal O(i)$ for $i \in \bb Z$. 
It is straightforward to observe the following:
\begin{enumerate}
    \item For $f_* \in \aut \bb P^1$, we have 
    \[
    \spc^{f_*} \perf \bb P^1 = \{\text{fixed points of $f:\bb P^1 \to \bb P^1$}\} \sqcup \bb Z \subset \bb P^1 \sqcup \bb Z.
    \] 
    \item For $\tau_ {\ecal O_{\bb P^1}(i) }\in \pic \bb P^1$ with $i\neq 0$, we have
    \[
    \spc^{\tau_{\ecal O(i)}} \perf \bb P^1 = \bb P^1 \subset \bb P^1 \sqcup \bb Z.
    \]
    \item For $[n] \in \bb Z$, we have
    \[
    \spc^{[n]}\perf \bb P^1 = \bb P^1 \sqcup \bb Z. 
    \]
\end{enumerate}
Therefore, for a general autoequivalence $\tau = \ecal O_{\bb P^1}(i) \tens_{\ecal O_\bb P^1} f_*(-)[n]$, we have 
\[
\spc^\tau \perf \bb P^1 = \begin{cases}
    \{\text{fixed points of $f:\bb P^1 \to \bb P^1$}\} \sqcup \bb Z  & \text{if $i = 0$;}\\ 
    \{\text{fixed points of $f:\bb P^1 \to \bb P^1$}\} & \text{if $i \neq 0$.}
\end{cases}
\]
Thus, we have a complete classification of polarizations on $\perf \bb P^1$ up to the geometry of its pt-spectrum. On the other hand, to the author's knowledge, it is still not known if we have a complete classification of tt-structures on $\perf \bb P^1$ up to tt-geometry. See \cite{castro2023spaces}*{\href{https://hal.science/tel-04018741v2/document}{Proposition 3.0.19}} and \cite{ito2023gluing}*{\href{https://arxiv.org/pdf/2309.08147.pdf}{Example 7.7 (iii)}} for some discussions. Also, note that not all tt-spectra are realized as fixed points of autoequivalences since the tt-spectrum for the quiver tt-structure on the Kronecker quiver consists of two points in $\bb Z$. We do not know if the converse is also false. That is, we do not know if given a polarization on $\perf X$, there exists a tt-structure on $\perf \bb P^1$ whose tt-spectrum agrees with the pt-spectrum.  
\end{example}
\begin{example}\label{example: preoplarization in perf E} Let us observe geometry of pt-spectra in the case of an elliptic curve $E$ over an algebraically closed field of characteristic zero. We will classify pt-spectra for generators and observe that a line bundle is $\tens$-ample if and only if the corresponding pt-spectrum agrees with the tt-spectrum (see the third bullet). On the other hand, there is a non-$\tens$-ample line bundle whose pt-spectrum agrees with the Balmer spectrum on all $k$-points. 

First, it is well-known that there is a short exact sequence
\[
1 \to \aut E \ltimes \pic^0 E \times \bb Z[2] \to \auteq \perf E \overset{\theta}{\to} \sl(2,\bb Z) \to 1
\]
where $\theta$ is given by the action on the image of the charge map $Z:K(E) \surj \bb Z^2,\ecal F \mapsto (\rank \ecal F,\deg \ecal F)$. Also recall from \cite{HO22}*{\href{https://arxiv.org/pdf/2112.13486}{Theorem 4.11}} that we have
\[
\spec_\vartriangle \perf E = \spec_{\tens_{\ecal O_E^\bb L}} \perf E \sqcup \bigsqcup_{(r,d) \in I} M_{r,d},
\]
where $I = \{(r,d) \in \bb Z_{>0} \times \bb Z \mid \gcd (r,d) = 1\}$ and 
\[
M_{r,d}:= \Phi_{\ecal U_{r,d}}\l(\spec_{\tens_{\ecal O_{M(r,d)}^\bb L}} \perf M(r,d)\r) 
\]
for the moduli space $M(r,d)$ of stable sheaves of rank $r$ and degree $d$ and the Fourier--Mukai transform $\Phi_{\ecal U_{r,d}}$ with respect to a universal family $\ecal U_{r,d}$. For convenience, let us write the subspace $$M_{0,1}:=  \spec_{\tens_{\ecal O_E^\bb L}} \perf E \subset \spec_\vartriangle \perf E.$$ It is easy to check
\[
M_{r,d}(k) = \{\cal S_{r,d}(\ecal F) \mid [\ecal F] \in M(r,d)(k)\},
\]
where we define 
\[
\cal S_{r,d}(\ecal F):= \bra{\ecal G \mid [\ecal G] \in M(r,d)(k), \ecal G \not \iso \ecal F} \subset \perf E. 
\]
Note the generic point of $M_{r,d}$ is given by $$\eta_{r,d}:=\bra{\ecal G \mid [\ecal G] \in M(r,d)(k)} \subset \perf E.$$ Also, recall that any object in $\perf E$ can be written as a direct sum of shifts of indecomposable coherent sheaves on $E$ and hence that any autoequivalence of $\perf E$ takes an indecomposable coherent sheaf to a shift of an indecomposable coherent sheaf. Hence, as a very rough estimate, for an autoequivalence $\tau \in \auteq \perf E$, we have
\[
\spc^\tau \perf E \subset \bigsqcup_{(r,d) \in I^\tau} M_{r,d},
\]
where $I^\tau := \{(r,d) \in I \cup \{(0,1)\}\mid \theta(\tau)(r,d) = \pm(r,d)\}$, noting that for a coherent sheaf $\ecal F$ on $E$, we have $(\rank \ecal F[n], \deg \ecal F[n]) = (-1)^n(\rank \ecal F, \deg \ecal F)$.

Now, let us consider the pt-spectrum corresponding to generators of $\auteq \perf E$. 

\begin{itemize}
    \item 

To get more explicit descriptions of generators of $\sl(2,\bb Z)$, note that $\theta$ maps $T_{k(x)} (= \tau_{\ecal O_E(x)})$ and $T_{\ecal O_E}$ to $\alpha:=\begin{pmatrix}
    1 & 0 \\ 1 & 1
\end{pmatrix}$ and $\beta:=\begin{pmatrix}
    1 & -1 \\ 0 & 1
\end{pmatrix}$, respectively, where $x$ is a closed point of $E$ and $T_\ecal E$ denotes the spherical twist with respect to an spherical object $\ecal E$. To visualize the corresponding pt-spectra, note that $T_{k(x)}$ maps the component $M_{r,d}$ to $M_{r,r+d}$ while $T_{\ecal O_E}$ maps the component $M_{r,d}$ to $M_{r-d,d}$ if $r-d \geq 0$ and to $M_{-r+d,- d}$ otherwise. Thus, we have
\[
\spc^{T_{k(x)}} \perf E = M_{0,1} 
\]
and 
\[
\spc^{T_{\ecal O_E}} \perf E = \bigsqcup_{r \in \bb Z_{>0}} M_{r,0}. 
\]
For the latter equality, we need to show that for any stable sheaf $\ecal F$ of rank $r$ and degree $0$, we have $$T_{\ecal O_E} \cal S_{r,0}(\ecal F)= \cal S_{r,0}(\ecal F).$$ To see this, note that we have 
\[
T_{\ecal O_E} \ecal F \iso \operatorname{cone}( (\h^0(E, \ecal F)\oplus \h^1(E,\ecal F)[-1])\otimes_\bb C \ecal O_E \overset{\sf{ev}}{\to} \ecal F)
\]
(e.g., cf. \cite{HuyBook}*{}) and there exists a unique stable sheaf $\ecal F_{r,0}$ of rank $r$ and degree $0$ with global sections (up to isomorphism) by \cite{atiyah1957vector}*{\href{https://math.berkeley.edu/~nadler/atiyah.classification.pdf}{Theorem 5}}. Therefore, for any stable sheaf $\ecal F \not \iso \ecal F_{r,0}$ of rank $r$ and degree $0$, we have $\h^0(E,\ecal F) = \h^1(E, \ecal F) = 0$ and hence $T_{\ecal O_E} \ecal F = \ecal F$, i.e., $T_{\ecal O_E} \cal S_{r,0}(\ecal F)= \cal S_{r,0}(\ecal F)$. Since $T_{\ecal O_E}$ induces an automorphism on the component $M_{r,0}$, we see that $T_{\ecal O_E} \cal S_{r,0}(\ecal F_{r,0})= \cal S_{r,0}(\ecal F_{r,0})$ as well. 

\item Consider $\tau = f_*$ for $f \in \aut E$. By definition we have
\[
\spc^\tau \perf E = \{\text{fixed points of $f$}\} \sqcup \bigsqcup_{(r,d) \in I}\l\{\cal S_{r,d}(\ecal F) \mid \ecal F \iso f_* \ecal F\r\} \cup \{\eta_{r,d}\}.
\]
Note a special case is when $f$ is a translation, in which case
\[
\spc^\tau \perf E= \bigsqcup _{r\in \bb Z_{>0}}M_{r,0}\sqcup \bigsqcup_{(r,d) \in I, d\neq 0} \{\eta_{r,d}\}.
\]
since the only homogeneous stable bundles are ones of degree $0$ by Atiyah's classification. Here, it is worth noting that the pt-spectra for $T_{\ecal O_E}$ and translations are the same while their images under $\theta$ are different. 

\item Consider $\tau = \tau_\ecal L (:=-\tens_{\ecal O_E}\ecal L)$ for $\ecal L \in \pic^0 E$. By definition we have 
\[
\spc^\tau \perf E = M_{0,1} \sqcup \bigsqcup_{(r,d) \in I} \{\cal S_{r,d}(\ecal F) \mid \ecal F \iso \ecal F \tens \ecal L\}\cup \{\eta_{r,d}\}. 
\]
By \cite{atiyah1957vector}*{\href{https://math.berkeley.edu/~nadler/atiyah.classification.pdf}{Corollary of Theorem 7}}, we see that:
\begin{itemize}
    \item if $\ecal L^{\tens r} \iso \ecal O_E$, then $\bigsqcup_{d \in \bb Z} M_{r,d} \subset \spc^\tau \perf E$ and
    \item if $\ecal L^{\tens r} \not \iso \ecal O_E$ for all $r>0$, then 
    \[
    \spc^\tau \perf E = M_{1,0}\sqcup \bigsqcup_{(r,d) \in I} \{\eta_{r,d}\}
    \]
\end{itemize}
In particular, a line bundle is $\tens$-ample if and only if the corresponding pt-spectrum agrees with the tt-spectrum. On the other hand, the second example is quite close to reconstruct the Balmer spectrum in the sense that the $k$-points of its pt-spectrum agree with those of the Balmer spectrum. In particular, the Balmer spectrum is realized as a unique maximal dimensional connected component of $\spc^\tau \perf E$. 
\end{itemize}
So far, we have observed the pt-spectra for the generators of $\auteq \perf E$, so for fun, let us consider some explicit examples:
\begin{itemize}
\item Set \[
\tau = \l(T_{\ecal O_E}T_{k(x)}\r)^3 = i^*[1],
\]
where $i$ denotes the inverse map $E \to E$ with respect to the identity element $e$ (cf. e.g., \cite{SeiTho01}). Clearly, we have 
\[
M_{0,1} \cap \spc^\tau \perf E = E_2, 
\]
where $E_2$ denotes the set of $2$-torsion points of $E$ (with respect to the fixed point $e$). To understand the action of $i^*$ on the moduli spaces, let us recall we have an isomorphism
\[
\det: M_{r,d} \overset{\sim}{\to} M_{1,d} = \pic^d E 
\]
so that $i^*\det \ecal F \iso \det i^* \ecal F$ for any $\ecal F \in M_{r,d}(k)$. Now, for a line bundle $\ecal L = \ecal O_E((d-1)e + p)$ of degree $d$, note that $i^*\ecal L \iso \ecal L$ if and only if $p \in E_2$. Therefore, setting $$\ecal F_{r,d}(p):=\det{}\inv(\ecal O_E((d-1)e+p)),$$ we see that 
\[
\spc^\tau \perf E = E_2 \sqcup \bigsqcup_{(r,d)\in I} \bigsqcup_{p \in E_2} \{\ecal F_{r,d}(p)\}, 
\]
where the generic points of all the connected components are included. 
\item Set
\[
\tau = \Phi_{\ecal U_{1,0}} \circ f_{{1,0}_*},
\]
where $f_{1,0}:E \to M_{1,0} (= \pic^0(E))$ is an isomorphism given by the fixed point $e$. It is easy to check $\theta(\tau) = \begin{pmatrix}
    0 & 1 \\ -1 & 0
\end{pmatrix}$ and hence we see
\[
\spc^\tau \perf E = \emp. \qedhere
\]
\end{itemize}
\end{example}
\subsection{Variants of polarizations}
Now, before considering reinterpretations of reconstructions, let us introduce the following natural generalizations of pt-categories, which will be used to deal with reconstructions of not necessarily quasi-projective schemes. 
\begin{definition}\label{def: weak ppt}
Let $\cal T$ be a triangulated category. 
    \begin{enumerate}
        \item For an endofunctor $\tau \in \End \cal T$, a pair $(\cal T, \tau)$ is said to be a \textbf{weakly pt-category}. In this context, an endofunctor $\tau$ may be interchangeably referred to as a \textbf{weak polarization}. We define the corresponding spectrum to be
        \[
        \spec^\tau \cal T : = (\spc^\tau \cal T, \ecal O_{\spc^\tau \cal T,\vartriangle}). 
        \]
        \item For a collection $\tau = \{\tau_i\}_{i \in I}$ of polarizations on $\cal T$, a pair $(\cal T,\tau)$ is said to be a \textbf{multi-pt-category} and the collection $\tau$ is said to be a \textbf{multi-polarization}. 
        We define the corresponding spectrum to be
        \[
        \spec^\tau \cal T : = (\spc^\tau \cal T, \ecal O_{\spc^\tau \cal T,\vartriangle}) 
        \]
        where we set
        \[
        \spc^\tau \cal T:= \bigcap_{i \in I} \spc^{\tau_i}\cal T. \qedhere
        \]
    \end{enumerate}
\end{definition}
To begin with, we have the following generalization of our guiding example (Proposition \ref{example:guiding}) for weakly pt-categories.
\begin{lemma}\label{lem: generalization to noetherian by wppt-cat}
        Let $X$ be a noetherian scheme and $\ecal E$ a split generator of $\perf X$. Write $\tens:= \tens_{\ecal O_X}^\bb L$ and set $\tau_{\ecal E} := -\tens \ecal E \in \End \perf X$. Then, we have
        \[
    |X| \iso \spc_{\tens} \perf X = \spc^{\tau_ \ecal E} \perf X \subset \spc_\vartriangle \perf X. 
    \]
    In particular, we have $X_\sf{red} \iso \spec^{\tau_\ecal E}\perf X$. 
\end{lemma}
    \begin{proof}
        First, take $\cal P \in \spc^{\tau_ \ecal E} \perf X \subset \spc_\vartriangle \perf X$. Then, since
        \[
        (\tau_ \ecal E)\inv(\cal P)= \{\ecal F \in  \perf X \mid \ecal F \tens \ecal E \in \cal P\}, 
        \]
        we have
        \[
        \cal P =  \{\ecal F \in  \perf X \mid \ecal F \tens \ecal E \in \cal P\}
        \]
        and hence for any $\ecal F \in \cal P$, we have $\ecal F \tens \ecal E \in \cal P$, i.e., $\cal P$ is a $\tens$-ideal by Lemma \ref{lem: split gen} as $\ecal E$ is a split generator. Since $\cal P$ is moreover a prime thick subcategory,  we have $\cal P \in \spc_\tens \perf X$ by Theorem \ref{thm: matsui prime thick ideal}.

        Conversely, take $\cal P \in \spc_\tens \perf X$. Since it is in particular a $\tens$-ideal, we have 
        \[
        \cal P \subset \{\ecal F \in  \perf X \mid \ecal F \tens \ecal E \in \cal P\}.
        \]
        On the other hand, since $\cal P$ is a prime $\tens$-ideal and it cannot contain a split generator (so $\ecal F \tens \ecal E \in \cal P$ implies $\ecal F \in \cal P$), we have 
        \[
        \cal P \supset \{\ecal F \in  \perf X \mid \ecal F \tens \ecal E \in \cal P\}.
        \]
        Therefore, 
        \[
        \cal P = \{\ecal F \in  \perf X \mid \ecal F \tens \ecal E \in \cal P\} = (\tau_ \ecal E)\inv(\cal P),
        \]
        i.e., $\cal P \in \spc^{\tau_ \ecal E}\perf X$. The latter claim follows from Theorem \ref{tt and triangle} (iii).
    \end{proof}
    By the same proof as above, we can generalize Proposition \ref{prop: main autoequiv} to the following claim:
    \begin{prop}\label{prop: tt as weak pt}
        Let $(\cal T, \tens)$ be an $M$-compatible tt-category with a split generator $\ecal E$ of $\cal T$. Then, we have
    \[
    \spc_\tens \cal T = \spc^{\tau_ \ecal E} \cal T \subset \spc_\vartriangle \cal T. \qedhere
    \]
    \end{prop}
    \begin{remark}\label{moduli}
        This proof tells us that the classifications of certain tt-structures on a fixed triangulated category $\cal T$ with a split generator up to Balmer spectrum is a part of classifications of weak polarizations, i.e., endofunctors up to weak pt-spectrum. To be more precise, recall in \cite{ito2023gluing}, it is defined that two tt-structures are \textbf{geometrically equivalent} if their Blamer spectra are contained in the Matsui spectrum and are the same subspace of the Matsui spectrum. Let us say two (weak) polarizations are \textbf{geometrically equivalent} if their (weak) pt-spectra are the same subspace of the Matsui spectrum. Then Proposition \ref{prop: tt as weak pt} shows that we can view the set $M^\sf{tt}_{\cal T}$ of geometric equivalence classes of tt-structures whose Balmer spectrum is contained in the Matsui spectrum is a subset of the set $M^\sf{wpt}_{\cal T}$ of geometric equivalence classes of weak polarizations by the following map:
        \[
        M_{\cal T}^\sf{tt} \to M_{\cal T}^\sf{wpt}; \quad \tens \mapsto -\tens\ecal E,
        \]
        which does not depend on the choice of a split generator. Moreover, Proposition \ref{example:guiding} tells us that for a quasi-projective variety $X$, the locus $M^\sf{tt}_{\perf X} \cap M^\sf{pt}_{\perf X}$ contains the tt-structure $\tens_{\ecal O_X}^\bb L$, so a better understanding of $M^\sf{tt}_{\perf X} \cap M^\sf{pt}_{\perf X}$ can lead to the full reconstruction of all Fourier--Mukai partners from $\perf X$, where $M^\sf{pt}_{\perf X}$ denote the set of geometric equivalence classes of polarizations.
    \end{remark}
    Now, let us use multi-polarizations to give another generalization of Proposition \ref{example:guiding}. First, let us recall the following notion:
    \begin{definition}
        Let $X$ be a quasi-compact quasi-separated scheme and $\{\ecal L_i\}_{i \in I}$ a collection of line bundles on $X$. Write $\tens:= \tens_{\ecal O_X}^\bb L$. We say $\{\ecal L_i\}_{i \in I}$ is an \textbf{ample family} if the following equivalent conditions hold (cf. \cite{sga6}*{6.II.2.2.3.}):
        \begin{enumerate}
            \item For any quasi-coherent $\ecal O_X$-module $\ecal F$ of finite type, there exist integers $n_i, k_i > 0$ such that $\ecal F$ is a quotient of $\oplus_{i \in I}(\ecal L_i^{\tens-n_i})^{\oplus k_i}$;
            \item There is a family of sections $f \in \Gamma (X,\ecal L_i^{\tens n})$ such that $X_f$ forms an affine open cover of $X$, where $X_f$ denotes the complement of the vanishing locus of $f$. 
        \end{enumerate}
        In condition (i), we can always take such a direct sum to be finite. If a quasi-compact and quasi-separated scheme $X$ admits an ample family, then $X$ is said to be \textbf{divisorial}. 
    \end{definition}
    \begin{prop}\label{prop: amply family genereated perf}
        Let $X$ be a divisorial scheme and let $\{\ecal L_i\}_{i\in I}$ be an ample family. Then, we have
        \[
        \bra{\ecal L^{\tens d} \mid d \in \bb Z^{\oplus I}} = \perf X. 
        \]
        where we are using multi-grading notations $d = (d_i)_{i \in I} \in \bb Z^{\oplus I}$ and $\ecal L^{\tens d} = \bigotimes_{i \in \I} \ecal L^{\tens d_i}$. 
    \end{prop}
    \begin{proof}
        The proof is essentially the same as \cite{Orlov_dimension}*{\href{https://arxiv.org/pdf/0804.1163}{Theorem 4}}, but let us provide a proof for readers' convenience. First, we take a locally free sheaf $\ecal F$ of finite rank on $X$ and we want to show $\ecal F \in \bra{\ecal L^{\tens d} \mid d \in I}$. By the definition of an ample family, there exists a bounded above complex $\ecal P$ each of whose term is a finite direct sums of line bundles $\ecal L_i^{\tens k}$ together with a quasi-isomorphism $\ecal P \to \ecal F$. Consider the brutal truncation $\sigma^{\geq -(m+1)} \ecal P$ for $m = \dim X$. Since the cone of the map $\sigma^{\geq -(m+1)} \ecal P \to \ecal F$ is isomorphic to $\ecal G[m+1]$ for some locally free sheaf $\ecal G$ and we have $\hom(\ecal F, \ecal G[m+1]) = \h^{m+1}(X, \ecal F^\vee \tens_{\ecal O_X} \ecal G) = 0$, we see that $\ecal F$ is a direct summand of $\sigma^{\geq -(m+1)} \ecal P$, i.e., $\ecal F \in \bra{\ecal L^{\tens d} \mid d \in I}$. Now, since any perfect complex over a divisorial scheme is quasi-isomorphic to a bounded complex of vector bundles (e.g., cf. \cite{thomason2007higher}*{\href{https://www.maths.ed.ac.uk/~v1ranick/papers/tt.pdf}{2.3.1.d}}), we see $\bra{\ecal L^{\tens d} \mid d \in I} = \perf X$. 
    \end{proof} 
    As a direct corollary, we obtain the following:
    \begin{theorem}\label{thm: divisorial reconstruction}
        Let $X$ be a noetherian divisorial scheme and take an ample family $\{\ecal L_i\}_{i \in I}$. Set $\tens:=\tens_{\ecal O_X}^\bb L$ and take the corresponding multi-polarization $\tau := \{\tau_{\ecal L_i}\}_{i \in I}$. Then, we have
        \[
        \spc_\tens \perf X = \spc^\tau \perf X \subset \spc_\vartriangle \perf X.
        \]
        In particular, we have $X_\sf{red} \iso \spec^\tau \cal T$. 
    \end{theorem}
    \begin{proof}
        The containment $\spc_\tens \perf X \subset \spc^\tau \perf X$ is clear. Conversely, if $\cal P \in \spc^\tau \perf X$, then for any $d \in \bb Z^{\oplus I}$, we have $\ecal L^{\tens d} \tens \cal P = \cal P$. Since $\bra{\ecal L^{\tens d}\mid d \in \bb Z^{\oplus I}} = \perf X$ by Proposition \ref{prop: amply family genereated perf}, we see that $\cal P$ is a $\tens$-ideal and hence $\cal P$ is a prime $\tens$-ideal by Theorem \ref{thm: matsui prime thick ideal} (ii). The latter claim follows from Theorem \ref{tt and triangle} (iii). 
    \end{proof}
    Finally, let us see mention that we may talk about an open subvariety of the original variety as a spectrum of a multi-polarization. 
\begin{example}\label{example: spherical twist}
    Let $S$ be a smooth projective surface over an algebraically closed field with a $(-2)$-curve $C \subset S$. Then recall that $\ecal O_C$ is a spherical object in $\perf S$ and there is the associated spherical twist $T_{\ecal O_S}$. Note in \cite{ito2023gluing}*{\href{https://arxiv.org/pdf/2309.08147}{Corollary 5.14}}, it is essentially observed that $$\spec_{\tens_S^\bb L} \perf S \cap \spec^{T_{\ecal O_C}}\perf S \iso S\setminus C,$$ so if we take a $\tens$-ample line bundle $\ecal L$ on $S$ and set $\tau = \{T_{\ecal O_C},\tau_\ecal L\}$, we have
    \[
     \spec^{\tau} \perf S \iso S \setminus C. 
    \]
    In particular, classifications of multi-polarizations up to their spectra must involve classifications of $(-2)$-curves. 
\end{example}

\section{Polarizations arising in nature}
In this section, we observe several situations where polarizations appear naturally.
\subsection{Reconstruction of varieties}
First, let us introduce the following notions to conveniently formulate/give proofs of several reconstructions.
\begin{definition}\label{def: pt-ample}
    Let $\cal T$ be a triangulated category and let $X$ be a noetherian scheme. A polarization $\tau$ is said to be \textbf{ample} (resp. \textbf{quasi-ample}) \textbf{in $X$} if 
    \begin{enumerate}
        \item $\spc^\tau \cal T$ has a unique maximal dimensional component $(\spc^\tau \cal T)^*$, and
        \item  there exists a tt-structure $\tens$ on $\cal T$ with a tt-equivalence $(\cal T, \tens) \simeq (\perf X, \tens_{\ecal O_X}^\bb L)$ such that 
        \[
        \spc_\tens \cal T = (\spc^\tau \cal T)^* \subset \spc_\vartriangle \cal T\quad \text{(resp. $\spc_\tens \cal T \subset (\spc^\tau \cal T)^*  \subset \spc_\vartriangle \cal T$)}.
        \]
    \end{enumerate}
        We define the same notions for weak polarizations and multi-polarizations. Moreover, we say a line bundle $\ecal L$ on $X$ is \textbf{pt-ample} if the polarization $\tau_\ecal L$ on $\perf X$ is ample in $X$ and define similarly for a collection of line bundles. Note by \cite{matsukawa2025spectrum}*{\href{https://arxiv.org/abs/2505.02724v1}{Theorem 2.11.}} that if we have the inclusion $\spc_\tens \cal T \subset \spc^\tau \cal T$, then it is open.
\end{definition}
\begin{example}
    Let $X$ be a noetherian scheme. 
    \begin{enumerate}
        \item If $\tau$ is (quasi-)ample in $X$, then so is any conjugation $\eta\circ \tau \circ \eta\inv$ for $\eta \in \auteq \perf X$. Similarly, if $\tau$ is (quasi-)ample in $X$, then so is any polarization $\eta\circ \tau \circ \eta\inv$ on a triangulated category $\cal T$ together with an exact equivalence $\eta: \perf X \simeq \cal T$.  
        \item A $\tens$-ample line bundle $\ecal L$ on $X$ is pt-ample by Proposition \ref{prop: main autoequiv}.
        \item On an elliptic curve, a non-torsion line bundle of degree $0$ is pt-ample, but not $\tens$-ample by Example \ref{example: preoplarization in perf E}. This is a motivating example for why we do not simply ask for $\spc^\tau \cal T = \spc_\tens \cal T$ in the definition. 
        \item For any line bundle $\ecal L$ on $X$, $\tau_\ecal L$ is quasi-ample in $X$. In particular, if $X$ is moreover a proper and Gorenstein variety, then the Serre functor is quasi-ample in $X$ and in any variety $Y$ with $\perf Y\simeq \perf X$.\qedhere         
    \end{enumerate}
\end{example}
Now, let us state a further generalization of the reconstruction of Bondal--Orlov and Ballard from the perspective of polarizations. 
\begin{theorem}\label{thm:Bondal--Orlov}
    Let $X$ be a Gorenstein proper variety of dimension $n$ and suppose $\spec^\ser \perf X$ is a separated scheme whose $n$-dimensional connected components (with reduced structure) are all isomorphic to each other. Then, the following assertions hold:
    \begin{enumerate}
        \item The variety $X$ can be reconstructed as an $n$-dimensional connected component of $\spec^\ser \perf X$ with reduced structure, which is constructed by using only the triangulated category structure of $\perf X$.
        \item If there exists a variety $Y$ with $\perf X \simeq \perf Y$, then $X \iso Y$. \qedhere
    \end{enumerate} 
\end{theorem}
Let us note that thanks to \cite{matsukawa2025spectrum}*{\href{https://arxiv.org/abs/2505.02724v1}{Theorem 2.11}} there is no need for any projectivity assumption and that Theorem \ref{thm:Bondal--Orlov} is indeed a strict generalization of \cite{ito2024new} by the following example (ii).
\begin{example} \label{example: tensor ample} \ 
\begin{enumerate}
    \item A Gorenstein proper variety with pt-ample canonical bundle satisfies the suppositions of Theorem \ref{thm:Bondal--Orlov} by definition.
    \item More specifically, a blow-up of a smooth projective surface with very ample canonical bundle over an algebraically closed field has a $\tens$-ample canonical bundle that is neither ample nor anti-ample by Proposition \ref{lem: tens-ample but not (anti)ample, blow up} and therefore we can still apply Theorem \ref{thm:Bondal--Orlov}.  
    \item By \cite{HO22}*{\href{https://arxiv.org/pdf/2112.13486}{Theorem 4.11}} (cf. Example \ref{example: preoplarization in perf E}), elliptic curves over algebraically closed field of characteristic zero also satisfy the suppositions. \qedhere
\end{enumerate}
\end{example}
Before giving a proof of Theorem \ref{thm:Bondal--Orlov}, let us make the following simple yet effective observation. 
\begin{theorem}\label{theorem: geometric in quasi-geometric}
    Let $\cal T$ be a triangulated category and suppose that there are proper, connected and reduced schemes $X,Y$ with $\perf X \simeq \perf Y \simeq \cal T$. Suppose there is a polarization $\tau$ on $\cal T$ that is ample in $X$ and quasi-ample in $Y$. Then, $\tau$ is ample in $Y$ and $X \iso Y$. The same claim follows for weak polarizations and multi-polarizations. 
\end{theorem}
\begin{proof}
    By supposition, there exist tt-structures $\tens_X$ and $\tens_Y$ such that
    \begin{itemize}
        \item $(\cal T,\tens_X)$ (resp. $(\cal T,\tens_Y)$) is tt-equivalent to $(\perf X, \tens_{\ecal O_X}^\bb L)$ (resp. $(\perf Y, \tens_{\ecal O_Y}^\bb L)$) and
        \item $(\spec^\tau \cal T)^* = \spec_{\tens_X} \cal T \iso X$ and $Y\iso \spec_{\tens_Y}\cal T \underset{\text{open}}{\subset} (\spec^\tau \cal T)^*$. 
    \end{itemize}
    Since $Y$ is proper and $X$ is separated, $Y$ is a closed subscheme of $X$ as well. Therefore, $Y$ is an open and closed subvariety of the connected scheme $X$ and hence $Y \iso \spec_{\tens_Y}\cal T =( \spec^\tau \cal T)^* \iso X$.
\end{proof}
Let us note that Theorem \ref{theorem: geometric in quasi-geometric} gives quick proofs for variants of \cite{FAVERO20121955}*{\href{https://www.sciencedirect.com/science/article/pii/S0001870812001211}{Theorem 3.8}}.
\begin{corollary}\label{cor: favero1}
    Let $\cal T$ be a triangulated category and suppose that there are proper, connected and reduced schemes $X,Y$ with $\perf X \simeq \perf Y \simeq \cal T$. Let $\ecal L$ be a pt-ample line bundle on $X$ and let $\ecal M$ be any line bundle on $Y$. Suppose there exists a pt-equivalence
    \[
    F:(\perf X, \tau_\ecal L) \overset{\sim}{\to} (\perf Y, \tau_\ecal M).
    \]
    Then, $\ecal M$ is a pt-ample line bundle on $Y$ and $X \iso Y$. The same claims follow for multi-polarizations by replacing $\ecal L$ with an amply family on $X$ (assuming $X$ is moreover divisorial) and $\ecal M$ with a collection of line bundles on $Y$. 
\end{corollary}
\begin{proof}
    Since $F\circ \tau_\ecal L \circ F\inv$ is a polarization on $\perf Y$ that is both ample in $X$ and quasi-ample in $Y$ (as $Y$ is noetherian), we are done by Theorem \ref{theorem: geometric in quasi-geometric}. 
\end{proof}
\begin{remark}In the original result by Favero, varieties are not assumed to be proper (only divisorial), but for our arguments to work, properness is essential. It is interesting to see if there is a way to show Favero's original result from the perspective of the Matsui spectrum.
\end{remark}
Now, with similar ideas, we give a short proof of the theorem of Bondal--Orlov and Ballard.
\begin{proof}[Proof of Theorem \ref{thm:Bondal--Orlov}] For 
    part (i), first note that $X \iso \spec_{\tens_X} \perf X$ is a connected, closed and open subscheme of $\spec^\ser \perf X$ by the same arguments as before. Thus, $\spec_{\tens_X} \perf X$ is an $n$-dimensional connected component of $\spec^\ser \perf X$ with reduced structure. Hence, by taking any $n$-dimensional connected component of $\spec^\ser \perf X$ with reduced structure, we can reconstruct $X$. 
    
    For part (ii), since $Y$ is automatically proper and Gorenstein (e.g., \cite{de2012reconstructing}*{\href{https://www.cambridge.org/core/journals/proceedings-of-the-edinburgh-mathematical-society/article/reconstructing-schemes-from-the-derived-category/7FA013C6F3E91CBFF0EF69631450DA1F}{Proposition 1.5}}), we can apply (i).
\end{proof}
\begin{question}
Is there a non-projective proper Gorenstein variety satisfying the supposition?
\end{question}
Finally, let us generalize yet another reconstruction result \cite{FAVERO20121955}*{\href{https://www.sciencedirect.com/science/article/pii/S0001870812001211}{Lemma 4.1}}. This result clarifies the interplay of automorphisms and line bundles as polarizations, which lead to certain rigidity.
\begin{theorem}\label{thm: favero}
    Let $X$ be a proper variety and take a polarization $\tau$ on $\perf X$. Suppose there are a variety $Y$ and a standard autoequivalence $$\sigma = f_*\circ \tau_{\ecal M}[i] \in \aut(Y) \ltimes \pic Y \times \bb Z[1]$$ together with a pt-equivalence
    \[
    \Phi:(\perf X, \tau) \simeq (\perf Y,\sigma)
    \]
    given by a Fourier--Mukai transform $\Phi=\Phi_\ecal K$ with Fourier--Mukai kernel $\ecal K \in \D(X\times Y)$. Then, the following assertions follow: 
    \begin{enumerate}
    \item If $\tau = \Phi_\ecal J$ for some $\ecal J \in \perf (X\times X)$ and if $\hom_{\perf X}(\ecal O_X, \tau^m \ecal O_X[l]) \neq 0$ for some $l, m \in \bb Z$, then $f^m$ has a fixed point. 
    \item If $\tau = \tau_\ecal L[n]$ with $\ecal L \in \pic X$ satisfying $\h^l(X, \ecal L^{\tens m}) \neq 0$ for some $l,m\in \bb Z$, then $f^m$ has a fixed point.
    \item If $f^m $ has a fixed point for some $m \in \bb Z$ and if every connected component of $\spec^{\tau^m}\perf X$ is isomorphic to $X$, then we have $f^m = \id_Y$ and $X \iso Y$. 
\end{enumerate}
In particular, if $\tau = \tau_\ecal L$ for a line bundle $\ecal L$ such that $\ecal L^{\tens m}$ is pt-ample with $\spc^{\tau_{\ecal L^{\tens m}}} \perf E$ connected and $\h^l(X,\ecal L^{\tens m}) \neq 0$ for some $l,m\in \bb Z$ (e.g., $\ecal L^{\tens m}$ is $\tens$-ample with global sections), then $f^m = \id_Y$ and $X \iso Y$. 
\end{theorem}
\begin{proof} Let us first recall some computations from \cite{FAVERO20121955}. Take $\ecal K' \in \D(X\times Y)$ so that $\Phi\inv = \Phi_{\ecal K'}$. Then, we get an equivalence $\Phi_{\ecal K \boxtimes \ecal K'}:\perf(X\times X) \simeq \perf(Y\times Y)$ such that if $\Phi:(\perf X, \Phi_\ecal S) \to (\perf Y, \Phi_\ecal T)$ is a pt-equivalence, then $\Phi_{\ecal K \boxtimes \ecal K'}(\ecal S)\iso \ecal T$. 
    \begin{enumerate}
        \item We proceed by mimicking the proof of \cite{FAVERO20121955}*{\href{https://www.sciencedirect.com/science/article/pii/S0001870812001211}{Lemma 4.1}}. Fix $l,m\in \bb Z$ so that $$\hom_{\perf X}(\ecal O_X, \tau^m \ecal O_X[l]) \neq 0$$ and write $\tau^m = \Phi_\ecal J$. Then, since $\Delta_*\tau^m\ecal O_X = \ecal J$ and $\sigma^m = \Phi_{(\id\times f^m)_*\ecal M_m[mi]}$ for a line bundle $\ecal M _m:= f_*\ecal M \tens \cdots \tens f_*^m \ecal M$, we get 
        \begin{align*}
            0\neq \hom_{X}(\ecal O_X, \tau^m \ecal O_X[l]) &\inj \hom_{X\times X}(\Delta_*\ecal O_X,\ecal J[l]) \\
            &\iso \hom_{Y\times Y}(\Phi_{\ecal K\boxtimes \ecal K'}(\Delta_*\ecal O_X), \Phi_{\ecal K\boxtimes \ecal K'}(\ecal J)[l])\\
            &\iso \hom_{Y\times Y}(\Delta_*\ecal O_Y,(\id_Y\times f^m)_*\ecal M_m[mi+l])\\
            &\iso \hom_{Y}(\ecal O_Y, \Delta^!(\id_Y\times f^m)_*\ecal M_m[mi+l]).
        \end{align*}
        Here, the first inclusion follows since the diagonal map $\Delta$ has a left inverse and hence $\Delta_* = \bb R \Delta_*$ is faithful on $\D_\sf{qc}$. Now, since the last term is the cohomology of a complex supported on the fixed locus $Y^{f^m}$ of $f^m$, we see that $f^m$ has a fixed point. 
        \item Immediately follows from part (i). 
        \item We may assume $m = 1$ since $\sigma^m = f_*^m \circ \tau_{\ecal M_m}[mi]$. Now, since $\spc_{\tens_Y^\bb L} \perf Y \subset \spc_\vartriangle
        \perf Y$ is an open inclusion, by taking fixed points of $\sigma$, we obtain 
        \[
        \emp \neq (\spc_{\tens_Y^\bb L} \perf Y)^\sigma \underset{\text{open}}{\subset} \spc^\sigma \perf Y \iso \spc^\tau \perf X. 
        \]
        Thus, by supposition $\dim (\spc_{\tens_Y^\bb L} \perf Y)^\sigma = \dim Y$. Now, note $$(\spc_{\tens_Y^\bb L} \perf Y)^\sigma = (\spc_{\tens_Y^\bb L} \perf Y)^{f_*\circ \tau_\ecal M} = (\spc_{\tens_Y^\bb L} \perf Y)^{f_*} \iso Y^f$$ since $\tau_\ecal M$ fixes every point in the Balmer spectrum. Therefore, we have $\dim Y^f = \dim Y$ and hence we necessarily have $f = \id_Y$. Then, we have $X \iso Y$ by Theorem \ref{theorem: geometric in quasi-geometric}.  
    \end{enumerate}
    The last claim follows from part (ii) and (iii). 
\end{proof}
There are some simplifications we can make to suppositions.
\begin{remark} Use the same notations as in Theorem \ref{thm: favero}.
    \begin{enumerate}
        \item If $X$ is projective, then $\Phi$ and $\tau$ are automatically of Fourier--Mukai type by \cite{ballard2009equivalences}*{\href{https://arxiv.org/pdf/0905.3148}{Theorem 1.2}}. 
        \item If $X$ is a variety all of whose automorphisms have a fixed point, then part (iii) simplifies. For example, any smooth proper complex variety $X$ with $\h^i(X,\ecal O_X) = 0$ for all $i>0$ falls into this category by the holomorphic Lefschetz fixed point formula. This includes smooth projective rationally connected varieties (e.g., smooth weak Fano varieties over an algebraically closed field of characteristic zero), Enriques surfaces and any smooth proper varieties birationally equivalent to those varieties. 
        \item On the other hand, the equidimensionality in part (iii) cannot be removed with our current proof, so we cannot apply this theorem with $\tau = \tau_\ecal L$ for a pt-ample line bundle $\ecal L$ with $(\spc^{\tau_\ecal L} \perf X)^* \subsetneq \spc^{\tau_\ecal L} \perf X$. \qedhere   
    \end{enumerate}
\end{remark}
Finally, let us leave another observation suggesting a further interplay between "positivity" of line bundles and fixed-point properties of automorphisms (see also \cite{FAVERO20121955}*{\href{https://www.sciencedirect.com/science/article/pii/S0001870812001211}{Remark 4.2}}). 
\begin{example}\label{example: pt-ample and fixed points}
   Let $E$ be an elliptic curve over an algebraically closed field of characteristic zero and take a line bundle $\ecal L$ on $E$ of degree $0$ with $\ecal L^{\tens n} \not \iso \ecal O_E$ for any $n>0$. In particular, $\ecal L$ is pt-ample. Then, as we observed in Example \ref{example: preoplarization in perf E}, $\ecal L$ is quite close to cutting out the Balmer spectrum while $\h^l(E, \ecal L^{\tens m}) = 0$ for all $l,m\in \bb Z$. Now, using the same notations as in Example \ref{example: preoplarization in perf E}, consider an equivalence
    \[
    \Phi_{\ecal P}: \perf \hat E \to \perf E,
    \]
    where $\hat E = M_{1,0} = \pic^0(E)$ is the dual of $E$ and $\ecal P = \ecal U_{1,0}$ is the Poincar\'e bundle. Then, it is easy to check that $\Psi:= \Phi_\ecal P\inv \circ \tau_\ecal L \circ \Phi_\ecal P:\perf \hat E\overset{\sim}{\to} \perf \hat E$ sends skyscraper sheaves at closed points to skyscraper sheaves at closed points and hence we can write $\Psi = f_* \circ \tau_\ecal M$ for some $f \in \aut \hat E$ and $\ecal M \in \pic \hat E$. Therefore, we have a pt-equivalence
    \[
    \Phi_{\ecal P}: (\perf \hat E,f_* \circ \tau_\ecal M)\simeq (\perf E,\tau_\ecal L).
    \] Note it is moreover clear that we have $f$ equals the translation ${t_{[\ecal L]}}$ on $\hat E$ by $[\ecal L] \in \hat E$, i.e., $$f = {t_{[\ecal L]}}: \hat E \overset{\sim}{\to}  \hat E,\quad [\ecal E] \mapsto [\ecal E \tens_E \ecal L].$$ 
    Hence, we have observed that the pt-ample line bundle $\ecal L$ fails to satisfy the supposition of Theorem \ref{thm: favero} (ii) and indeed $f = {t_{[\ecal L]}}$ has no periodic point. On the other hand, if we take $\ecal L$ to be a line bundle of degree $0$ with $\ecal L^{\tens n} \iso \ecal O_E$ for some $0 \neq n \in \bb Z$, then $f^n = t_{[\ecal L^n]} = t_{[\ecal O_E]} = \id_E$ and $\spc^{\tau_{\ecal L^{\otimes n}}} \perf E = \spc_\vartriangle \perf E$ is equidimensional of dimension $1$, so we can apply Theorem \ref{thm: favero} (iii) while $\ecal L$ is not pt-ample. 
\end{example}
\subsection{Geometric realization of noncommutative projective schemes and relations to Iitaka fibrations}\label{subsec: Iitaka}
We observe that pt-spectra can be thought of as  geometric realization of (derived) noncommutative projective schemes in the sense of \cite{artin1994noncommutative}. Along the way, we will also make some observations on relations with birational geometry. First, let us quickly recall materials from \cite{artin1994noncommutative}. 
\begin{definition}
    Let $A$ be a right noetherian $\bb N$-graded algebra (over $k$). The \textbf{noncommutative projective scheme} ${\proj}^\sf{nc} A$ associated to $A$ is a triple $(\operatorname{\sf{qgr}} A, s_A, A)$, where 
    \begin{itemize}
        \item the symbol  $\operatorname{\sf{qgr}} A$ denotes the ($k$-linear) abelian category obtained as the quotient of the abelian category of finitely generated $\bb Z$-graded (right) $A$-modules by the Serre subcategory of finitely generated torsion $\bb Z$-graded $A$-modules,
        \item the symbol $s_A$ denotes the autoequivalence of $\operatorname{\sf{qgr}} A$ given by shifting gradings by one, and 
        \item the symbol $A$ denotes the image of $A$ in $\operatorname{\sf{qgr}} A$ by abuse of notation.
    \end{itemize}
    Similarly, let $\operatorname{\sf{QGr} }A$ denote the quotient category of the abelian category of $\bb Z$-graded (right) $A$-modules by the Serre subcategory of torsion $\bb Z$-graded $A$-modules. Note $\operatorname{\sf{QGr}}A$ is the ind-completion of $\operatorname{\sf{qgr}} A$. 
\end{definition}
The following classical results are their guiding examples (see for example \cite{Har77}*{Exercise II.5.9.}):
\begin{theorem}\label{thm: artin-zhang guiding}
    Let $A$ be a finitely generated $\bb N$-graded commutative algebra generated in degree $1$. 
    \begin{enumerate}
        \item The functor
    \[
    \Gamma: \operatorname{\sf{Coh}}\proj A \to \operatorname{\sf{qgr}} A, \quad \ecal M \mapsto \oplus_{n \in \bb Z} \Gamma(\proj A, \ecal M(n))
    \]
    is an equivalence of abelian categories, where $\operatorname{\sf{Coh}} \proj A$ denotes the abelian category of coherent sheaves on $\proj A$. Under this equivalence, the autoequivalence $-\tens_{\ecal O_{\proj A}} \ecal O_{\proj A}(1)$ on $\operatorname{\sf{Coh}} (\proj A)$ corresponds to the autoequivalence $s_A$ on $\operatorname{\sf{qgr}}A$. Furthermore, $\Gamma$ naturally extends to an equivalence
    \[
    \Gamma: \operatorname{\sf{QCoh}}\proj A \simeq \operatorname{\sf{QGr} A},
    \]
    where $\operatorname{\sf{QCoh}}\proj A$ denotes the abelian category of quasi-coherent sheaves on $\proj A$. 
    \item The functor $\Gamma$ above induces an isomorphism
\[
\proj A = \proj\l (\oplus_{n\geq 0}\Gamma(\proj A, \ecal O_{\proj A}(n))\r) \iso \proj \l (\oplus_{n\geq 0}\hom_{{\operatorname{\sf{qgr}}A}}{\l(A,s_A^nA\r)}\r)
\]
of schemes. In particular, note that the right hand side is constructed precisely by the data of $\proj^\sf{nc} A$. \qedhere
\end{enumerate}
    
\end{theorem}
Therefore, for a general right noetherian $\bb N$-graded algebra $A$, we may think of $\proj^\sf{nc} A$ as a "(noncommutative) space" whose abelian category of coherent sheaves (resp. quasi-coherent sheaves) is $\operatorname{\sf{qgr}} A$ (resp. $\operatorname{\sf{QGr}} A$). Now, to make comparisons of pt-spectra with noncommutative projective schemes and to see some relations with birational geometry, let us introduce some notions and recall some facts on the Iitaka fibration. 
\begin{notation}
For a pt-category $(\cal T, \tau)$ and a fixed object $c \in \cal T$, define the $\bb N$-graded (not necessarily commutative) algebra
    \[
    \cal R_{\tau,c}:= \oplus_{n \geq 0}\hom_{\cal T}(c, \tau^n(c)),
    \]
    where multiplications are given by compositions under identifications $$\tau^m:\hom_\cal T(c,\tau^n(c)) \iso \hom_\cal T(\tau^m(c),\tau^{m+n}(c)).$$ For a pt-category $(\perf X,\tau_\ecal L)$ given by a line bundle $\ecal L$ on a normal projective variety $X$, note that this construction gives the \textbf{section ring}
    \[
    \cal R_\ecal L:= \cal R_{\tau_{\ecal L},\ecal O_X} \iso \oplus_{n \geq 0} \Gamma(X,\ecal L^{\tens n})
    \]
    of $\ecal L$, where multiplications are given by multiplication of sections and hence are commutative. Note if $\cal R_{\ecal L}$ is a finitely generated $k$-algebra (e.g., if $\ecal L$ is semi-ample (\cite{lazarsfeld2017positivity}*{Example 2.1.30}) or if $\ecal L = \omega_X$ for a smooth projective variety $X$ (\cite{birkar2010existence})) and if the Iitaka dimension of $\ecal L$ is positive, then there is a rational map
    \[
    \phi_{\ecal L}: X \ratmap \proj \cal R_\ecal L
    \]
    that is birationally equivalent to the natural rational maps $$\phi_{|\ecal L^{\tens n}|}: X \ratmap Y_n:=\phi_{|\ecal L^{\tens n}|}(X) \subset \bb P\h^0(X, \ecal L^{\tens n})$$ for all large enough $n$ with $\ecal L^{\tens n}$ being effective, in which case $Y_n \iso \proj \cal R_\ecal L$ by finite generation of $\cal R_\ecal L$ (cf. \cite{lazarsfeld2017positivity}*{Theorem 2.1.33}). 
    The rational map $\phi_\ecal L$ is called the \textbf{Iitaka fibration} and $\dim \proj \cal R_\ecal L$ equals the Iitaka dimension of $\ecal L$. Furthermore, if $\ecal L$ is moreover {big} (resp. semi-ample), then $\phi_\ecal L$ is birational (resp. defined everywhere) (cf. \cite{lazarsfeld2017positivity}*{Definition 2.2.1. (resp. Theorem 2.1.27.)}).

    As a special case, when $X$ is a smooth projective variety, the construction above gives the \textbf{canonical ring} $\cal R_X: = \cal R_{\omega_X}$ of $X$ and the \textbf{canonical model} $X^\sf{can}:= \proj \cal R_X$ of $X$. By the results above, if the Kodaira dimension of $X$ (i.e., the Iitaka dimension of $\omega_X$) is positive, we have a natural rational map
    \[
    \phi_X: X \ratmap X^\sf{can}
    \]
    where $\dim X^\sf{can}$ equals the Kodaira dimension of $X$ and $\phi_X$ is birational when $X$ is \textbf{of general type} (i.e., $\omega_X$ is big). The canonical ring and the canonical model are important in birational geometry since they are birational invariants of smooth projective varieties. 
\end{notation}
Now, let us make our guiding observation. 
\begin{prop}\label{prop: nps = pts}
     Let $X$ be a projective scheme and take an ample line bundle $\ecal L$. Then, we have isomorphisms
     \[
     \spec^{\tau_\ecal L} \perf X \iso X \iso (\proj \cal R_{\ecal L})_\sf{red}. 
     \]
     
     Moreover, we have a pt-equivalence
     \[
     (\perf X,\tau_{\ecal L}) \simeq (\perf \proj^\sf{nc} \cal R_{\ecal L}, s_{\cal R_\ecal L})
     \]
     where $\perf \proj^\sf{nc} \cal R_{\ecal L}$ is the subcategory on compact objects in the unbounded derived category $\D(\operatorname{\sf{QGr}} \cal R_\ecal L)$ of $\operatorname{\sf{QGr}} \cal R_\ecal L$ and $s_{\cal R_\ecal L}$ denotes the restriction of the derived functor of $s_{\cal R_\ecal L}$ on $\D(\operatorname{\sf{QGr}} \cal R_\ecal L)$ to compact objects by slight abuse of notation. 
\end{prop}
\begin{proof}
    For the first claim, the first isomorphism is Proposition \ref{example:guiding} and the second isomorphism is classical (e.g., cf. \cite{GW10}*{Corollary 13.75}). The last claim follows from Theorem \ref{thm: artin-zhang guiding} (i) and the fact that the compact objects in the unbounded derived category of quasi-coherent sheaves on a quasi-compact quasi-separated scheme agrees with perfect complexes (e.g., cf. \cite{stacks-project}*{\href{https://stacks.math.columbia.edu/tag/09M8}{Tag 09M8}}). 
\end{proof}

Therefore, given a noncommutative projective scheme $\proj^{\sf{nc}} A$, we may think of the ringed space
\[
\spec^{s_A} \perf \proj^\sf{nc} A
\]
as its geometric realization with reduced structure. The following is one obvious future direction, which we do not pursue further in this paper.
\begin{question}
    How much information of $A$ and $\perf \proj^\sf{nc} A$ is retained in $\spec^{s_A} \perf \proj^\sf{nc} A$. Moreover, how rich is the geometry of $\spec^{s_A} \perf \proj^\sf{nc} A$ when $A$ is noncommutative? 
\end{question}
\begin{remark}
    We can learn several more observations from Proposition \ref{prop: nps = pts}.
    \begin{enumerate}
        \item It is natural to ask if we can also recover the non-reduced structure sheaf on $\spc^{\tau_\ecal L}\perf X$, corresponding to $\ecal O_X$. Not surprisingly, it is possible if we use the extra data of an object $c:= \ecal M[n] \in \perf X$ for any line bundle $\ecal M$ on $X$ and any $n \in \bb Z$. Indeed, the same construction as the structure sheaf of the tt-spectrum where we use $c$ instead of the tensor unit object (cf. the last paragraph of Construction \ref{construction: structure sheaf of tt-spectra}), we can define the sheaf $\ecal O_{\tau_{\ecal L},c}:=\ecal O_{\spc^{\tau_{\ecal L}}\perf X,c}$ of rings on $\spc^{\tau_\ecal L} \perf X$ and moreover by construction we have an isomorphism
        \[
        \spec^{\tau_\ecal L}_{c} \perf X:= (\spc^{\tau_\ecal L} \perf X, \ecal O_{\tau_{\ecal L},c}) \iso \spec_{\tens_X} \perf X \iso \proj \cal R_{\ecal L}.
        \]
        Note on the right hand side, we construct a ringed space using both $\tau_\ecal L$ and $\ecal O_X$ simultaneously while on the left hand side, we construct the underlying topological space and the structure sheaf separately, which indicates different nature of those two constructions.
        
        Now, since a noncommutative projective scheme $\proj^\sf{nc} A$ comes with a triple $(\perf \proj^\sf{nc} A,s_A,A)$, we can construct a ringed space $$\spec_A^{s_A} \perf \proj^\sf{nc} A := (\spc^{s_A} \perf \proj^\sf{nc} A, \ecal O_{s_A,A})$$ as above, which may be thought of as a more natural geometric realization of $\proj^\sf{nc} A$. Note on the other hand that $\ecal O_{s_A,A}$ is in general a sheaf of noncommutative algebras. 
        \item For a line bundle $\ecal L$ on a projective variety $X$, $\spec^{\tau_\ecal L}\perf X$ and $\proj \cal R_\ecal L$ do not necessarily agree. For example, if $X$ is a smooth projective variety, then we have
        \[
        \spec^{\tau_{\omega_X}} \perf X = \spec^\ser \perf X
        \]
        while we have $\proj \cal R_{\omega_X} = \proj \cal R_{X} = X^\sf{can}$. To be more concrete, let $S'$ be a blow-up of a smooth projective surface $S$ with very ample canonical bundle at a point as in Proposition \ref{lem: tens-ample but not (anti)ample, blow up}. Since the canonical model is birational invariant, we may relate them via the Iitaka fibration
        \[
        \spec^{\tau_{\omega_{S'}}} \perf S' \iso S' \ratmap S \iso \proj \cal R_{S} \iso \proj \cal R_{S'}.
        \]
        Therefore, in general, the pt-spectra are supposed to remember more birational geometric data than the Proj construction.  \qedhere
    \end{enumerate}
\end{remark}
Combining (i) and (ii) for a smooth projective variety $X$ with a good enough line bundle $\ecal L$, the Iitaka fibration provides a way of naturally comparing two ringed spaces 
$\spec^{\tau_\ecal L}_{\ecal O_X} \perf X$ and $\proj \cal R_{\ecal L}$ constructed from the data of the same triple $(\perf X, \tau_{\ecal L}, \ecal O_X)$ via the rational map
\[
\spec^{\tau_\ecal L}_{\ecal O_X} \perf X\supset \spec_{\tens_X}\perf X \iso X \ratmap \proj \cal R_{\ecal L}  = \spec^{s_{\cal R_\ecal L}}_{\cal R_\ecal L} \perf \proj^\sf{nc} \cal R_\ecal L. 
\]
Now, a natural question is if we can do this process purely categorically. Namely:
\begin{question}
    Let $X$ be a smooth projective variety and $\ecal L$ be a nice enough (e.g., big) line bundle on $X$. Can we construct a pt-functor 
    \[
    (\perf \proj^\sf{nc} \cal R_\ecal L,s_{\cal R_\ecal L}) \to (\perf X,\tau_\ecal L)
    \]
    (or the other way around) that induces a "rational" map
    \[
    \spec^{\tau_\ecal L}_{\ecal O_X} \perf X \ratmap \proj \cal R_\ecal L
    \]
    agreeing with the Iitaka fibration? For $\cal L = \omega_X$, is the corresponding rational map $\spec^\ser \perf X \ratmap X^\sf{can}$ universal in the sense that for any Fourier--Mukai partner $Y$ that is birationally equivalent to $X$, it is birationally equivalent to $\spec^\ser \perf Y \ratmap Y^\sf{can}$? Such categorification of Iitaka fibrations would align with broader objectives of the noncommutative minimal model program (\cite{halpern2023noncommutative}).  
\end{question}
\begin{remark}
    Such a map can be understood as an analogue of comparison maps in tensor triangulated geometry (cf. \cite{balmer2010spectra}). 
\end{remark}

\subsection{Constructions of mirror partners as spectra of Fukaya categories}\label{subsec: mirror}
Homological mirror symmetry predicts that there is an equivalence between the bounded derived category $\D^b _\coh (X)$ of coherent sheaves on a certain variety $X$ over $\bb C$ and a certain category $\operatorname{\sf{Fuk}}(M,\omega)$ associated to a symplectic manifold $(M,\omega)$ possibly with additional data. The notation $\operatorname{\sf{Fuk}}(M,\omega)$ is based on the Fukaya category, but it could denote the Fukaya-Seidel category, the (partially) wrapped Fukaya category, the category of constructible sheaves with prescribed microsupport, etc. If there is such an equivalence $\D^b_\coh(X) \simeq \operatorname{\sf{Fuk}}(M,\omega)$, we say $X$ and $(M,\omega)$ are \textbf{mirror partners}. On the other hand, we know that in general there are non-isomorphic Fourier--Mukai partners, so we have the following natural questions.
\begin{question}
    Given a symplectic manifold $(M,\omega)$, is there a canonical choice of an algebraic mirror partner? What kind of data of the symplectic manifold is responsible for determining its mirror partner(s)? 
\end{question}
If we suppose that a mirror partner of $(M,\omega)$ is a divisorial variety $X$ (e.g., a smooth variety or a quasi-projective variety), then from our perspective those questions are equivalent to asking if we can canonically associate to $\sf{Fuk}(M,\omega)$ a multi-polarization corresponding to an ample family on $X$ (cf. Theorem \ref{thm: divisorial reconstruction}), using purely symplecto-geometric data of $(M,\omega)$. In other words:
\begin{question}
    Is there a symplecto-geometric way of constructing a (multi)polarization $\tau$ on $\operatorname{\sf{Fuk}} (M, \omega)$ such that $\spec^\tau \operatorname{\sf{Fuk}}(M,\omega)$ is an algebraic mirror partner of $(M,\omega)$, possibly using additional data? How canonical is such a mirror partner? 
\end{question}

Indeed, in some cases, the answers are yes. In this section, we will follow the philosophy of the SYZ picture in the sense that Lagrangian sections of the torus fibration $M \to B$ are supposed to correspond to line bundles on (the compactification of) $X$ (cf. \cite{abouzaid2009morse}). Before looking into examples, let us mention relations with existing works.
\begin{itemize}
        \item Note that with our approach, we can explicitly construct algebraic mirror partners as pt-spectra, so our picture can lead to predictions of algebraic mirror partners given symplectic data. Note the construction of symplectic mirrors as certain spectra of algebraic mirrors is done for log Calabi-Yau surfaces in \cite{gross2015mirror}. We will discuss more on homological mirror symmetry for log Calabi-Yau surfaces later. 
        \item It is also natural to ask if we can geometrically construct a symmetric monoidal structure $\tens_M$ on $\sf{Fuk}(M,\omega)$ so that under homological mirror symmetry, we have 
        \[
        X \iso \spec_{\tens_X} \perf X \iso \spec_{\tens_M} \sf{Fuk}(M,\omega). 
        \] 
\end{itemize}
 Let us mention constructions of two such monoidal structures which are related to the context of polarizations. 
    \begin{example}\      
        \begin{enumerate}
            \item In \cite{subotic2010monoidal}*{\href{https://natebottman.github.io/docs/A_monoidal_structure_for_the_Fukaya_category.pdf}{Theorem 5.0.34.}}, a symmetric monoidal structure $\hat \tens_M$ is constructed on the generalized Donaldson-Fukaya category $\sf{Don}^\sharp(M)$ (c.f. \cite{subotic2010monoidal}*{\href{https://natebottman.github.io/docs/A_monoidal_structure_for_the_Fukaya_category.pdf}{Definition 4.0.27}}) for a smooth Lagrangian torus fibration $\pi: M \to B$ over a compact connected base, which restricts to fiberwise addition on the level of objects that are Lagrangian sections of the fibration. It is also shown that when $M$ is a $2$-torus, homological mirror symmetry (\cite{polishchuk1998categorical}) gives a monoidal equivalence $$(\perf X, \tens_X) \simeq (\sf{Don}^+(M), \hat \tens_M)$$ where $\sf{Don}^+(M)$ denotes the extended Donaldson-Fukaya category (\cite{subotic2010monoidal}*{\href{https://natebottman.github.io/docs/A_monoidal_structure_for_the_Fukaya_category.pdf}{Theorem 7.0.45.}}). It is recently announced in \cite{abouzaid2024focus} that the construction of this monoidal structure can be extended to singular cases, more precisely, to almost toric fibrations only with focus-focus singularities in the sense of \cite{symington2003four}*{\href{https://arxiv.org/pdf/math/0210033}{Definition 4.2, Definition 4.5.}}, while it seems to be still unknown if those monoidal structures correspond to $\tens_X$ on $\perf X$ under homological mirror symmetry. We will later mention when such a monoidal structure $\hat \tens_M$ can be expected to be at least geometrically equivalent (cf. Remark \ref{moduli}) to $\tens_X$ on $\perf X$ (Proposition \ref{prop: mirror partner as spectra} and Remark \ref{rem: mirror consequences} (iv)) under homological mirror symmetry, i.e., 
            \[
            X \iso \spec_{\tens_X}\perf X \iso \spec_{\hat \tens_M} \sf{Don}^+(M). 
            \]
            
            \item In \cite{hanlon2019monodromy}, Hanlon introduced the \textbf{monomially admissible Fukaya-Seidel category} $\cal F_\Delta(W_\Sigma)$ for the Landau-Ginzburg model $((\bb C^*)^n,W_\Sigma)$, where $W_\Sigma$ is the Hori-Vafa superpotential associated to a fan $\Sigma \subset \bb R^n$ defining a smooth proper toric variety $X_\Sigma$ and $\Delta$ is additional data called a monomial division associated to a fixed toric K\"ahler form on $(\bb C^*)^n$ and a fixed moment map $\mu:(\bb C^*)^n \to \bb R^n$ (\cite{hanlon2019monodromy}*{\href{https://arxiv.org/pdf/1809.06001v2}{Definition 2.1, Definition 3.18.}}). Roughly speaking, the objects are Lagrangians in $(\bb C^*)^n$ subject to admissibility conditions given by $\Delta$. Now, for the subcategory 
            \[
            \cal F^s_\Delta(W_\Sigma) \subset \cal F_\Delta(W_\Sigma)
            \]
            of Lagrangian sections of the torus fibration $\mu$, Hanlon constructed a symmetric monoidal structure $\tens_{\Delta, \Sigma}$ on the homotopy category $\ho(\cal F^s_\Delta (W_\Sigma))$ so that there is a symmetric monoidal equivalence
            \[
            (\operatorname{\sf{Lin}} X_\Sigma, \tens_{X_\Sigma}) \simeq (\ho(\cal F^s_\Delta (W_\Sigma)), \tens_{\Delta, \Sigma})
            \]
            where $(\operatorname{\sf{Lin}} X_\Sigma, \tens_{X_\Sigma})$ denotes the symmetric monoidal full subcategory of $(\perf X_\Sigma, \tens_{X_\Sigma})$ on line bundles \cite{hanlon2019monodromy}*{\href{https://arxiv.org/pdf/1809.06001v2}{Corollary 4.6.}}. Since $X_\Sigma$ is assumed to be smooth (in particular, divisorial) and hence $\operatorname{\sf{Lin}} X$ split-generates $\perf X = \D^b_\coh(X)$ by Proposition \ref{prop: amply family genereated perf}, any extension $\hat \tens_{\Delta,\Sigma}$ of $\tens_{\Delta,\Sigma}$ to $\D^\pi \cal F^s_\Delta(W_\Sigma)$ agrees with $\tens_X$ on objects under homological mirror symmetry, i.e., we have
            \[
            X_\Sigma  \iso \spec_{\tens_{X_\Sigma}}\perf X_{\Sigma} \iso \spec_{\hat \tens_{\Delta,\Sigma}} \D^\pi \cal F^s_\Delta (W_\Sigma). 
            \]
            In fact, the symmetric monoidal structure $\tens_{\Delta, \Sigma}$ on $\ho(\cal F^s_\Delta (W_\Sigma)) \simeq \operatorname{\sf{Lin}}X$ is constructed by identifying multi-polarizations on $\cal F_\Delta(W_\Sigma)$ corresponding to $\pic X$ under the equivalence, so it is indeed more direct to construct $X_\Sigma$ as a pt-spectrum instead of as a tt-spectra. \qedhere
        \end{enumerate}   
    \end{example}

In the rest of this section, we will consider spherical twists and some other classes of autoequivalences together with their corresponding pt-spectra. We will particularly focus on homological mirror symmetry for log Calabi-Yau surfaces as in \cite{hacking2021symplectomorphisms} and \cite{hacking2023homological}. Our main result is Proposition \ref{prop: mirror partner as spectra}, which reformulates results in \cite{hacking2021symplectomorphisms} and realizes algebraic mirror partner as pt-spectra of Fukaya categories. Readers with familiarity in \cite{hacking2021symplectomorphisms} and \cite{hacking2023homological} can safely skip to Proposition \ref{prop: reform hk}. 
\begin{notation}
{In the rest of this section, we will assume an algebraic mirror partner $X$ is smooth} so that homological mirror symmetry means an equivalence $\perf X = \D^b_\coh (X) \simeq \operatorname{\sf{Fuk}}(M,\omega)$ for simplicity. However, even without the assumption, we can actually generalize the arguments below to any (divisorial) variety $X$ by recalling $\perf X$ can be constructed from $\D^b_\coh(X)$ as locally finite homological functors on $\D^b_\coh(X)$ by \cite{neeman2022finite}*{\href{https://arxiv.org/pdf/2211.06587}{Theorem 6.6.(2)}} (or by \cite{rouquier2008dimensions}*{\href{https://www.math.ucla.edu/~rouquier/papers/dimension.pdf}{Corollary 7.51.}} if $X$ is moreover projective). Hence, in those cases, we can replace $\operatorname{\sf{Fuk}}(M,\omega)$ with the corresponding triangulated category $\operatorname{\sf{Fuk}}(M,\omega)^\vee$ of locally finite homological functors so that we have $\perf X \simeq \operatorname{\sf{Fuk}}(M,\omega)^\vee$. Indeed, by construction, any pt-equivalence $$(\D^b _\coh(X),\tau)\simeq (\operatorname{\sf{Fuk}}(M,\omega),\sigma)$$ restricts to a pt-equivalence \[(\perf X, \tau|_{\perf X}) \simeq (\operatorname{\sf{Fuk}}(M,\omega)^\vee,\sigma|_{\operatorname{\sf{Fuk}}(M,\omega)^\vee}).\qedhere\]
\end{notation}
Now, let us first begin with geometrically understanding some classes of spherical twists on $\sf{Fuk}(M,\omega)$. 
\begin{example}[Dehn twists and spherical twists] To a \textbf{framed exact Lagrangian sphere} $S$ in an exact symplectic manifold $(M,\omega)$ (i.e., an exact Lagrangian submanifold with a fixed diffeomorphism to a sphere), we can associate an automorphism $\tau_S$ of $(M,\omega)$, called the \textbf{Dehn twist} along $S$ (cf. e.g., \cite{seidel2008fukaya}*{(16c)}), which induces the autoequivalence ${\tau_S}_*$ on the Fukaya category $\operatorname{\sf{Fuk}}(M,\omega)$. In \cite{SeiTho01}, it is observed that a Lagrangian sphere is a spherical object in the the split-closed derived Fukaya category $\D^\pi \operatorname{\sf{Fuk}}(M,\omega)$ (\cite{SeiTho01}*{1c.}) and thus we can construct the associated spherical twists $T_S$. Now, by \cite{seidel2008fukaya}*{Corollary 17.17.}, we see that over $\D^\pi \operatorname{\sf{Fuk}}(M,\omega)$, the autoequivalence ${\tau_S}_*$ agrees with the spherical twist $T_S$ on objects and hence, independently of choices of a framing on $S$, we have
\[
\spec^{{\tau_S}_*}\D^\pi \operatorname{\sf{Fuk}}(M,\omega) = \spec^{T_S}\D^\pi \operatorname{\sf{Fuk}}(M,\omega). 
\]
On the other hand, since pt-spectra given by spherical twists can behave in various ways (cf. Example \ref{example: preoplarization in perf E} and Example \ref{example: spherical twist}), we cannot expect Dehn twists give a multi-polarization that constructs an algebraic mirror. In particular, note that over a variety of dimension $\geq 2$, spherical twists will never agree with tensoring with line bundles by combining \cite{ito2023gluing}*{\href{https://arxiv.org/abs/2309.08147}{Corollary 5.15.}} with the fact that skyscraper sheaves on a (smooth) variety of dimension $\geq 2$ is not spherical. Nevertheless, it will be interesting to understand the relation of the geometry of pt-spectra and Lagrangian spheres in the case of elliptic curves, where we already have explicit computations.
\end{example}
As we just observed above, we are interested in automorphisms of a symplectic manifold since they are supposed to induce autoequivalences of $\operatorname{\sf{Fuk}}(M,\omega)$. In \cite{hacking2021symplectomorphisms}, more classes of automorphisms, called {Lagrangian translations} and nodal slide recombinations, are defined and studied for Weinstein $4$-manifolds that are mirrors to log Calabi-Yau surfaces. Since Lagrangian translations correspond to tensoring with line bundles and thus are particularly relevant to us, let us explain them more in detail. 
\begin{definition}
    We say a smooth quasi-projective surface $U$ is a \textbf{log Calabi-Yau surface with maximal boundary} if there exist a smooth rational projective surface $Y$ and a singular nodal curve $D\in |-K_Y|$ such that $U = Y \setminus D$. In particular, $D$ needs to be either an irreducible rational nodal curve or a cycle of smooth rational curves. We may also refer to such a pair $(Y,D)$ as a log Calabi-Yau surface with maximal boundary. 
\end{definition}
\begin{example}
    A most basic example is when $Y$ is a smooth projective toric surface and $D$ is its torus boundary divisor consisting of a cycle of rational curves. Note in this case we have $U = Y\setminus D$ is a torus $\bb G_m^2$. On the other hand, when a smooth projective toric surface $Y$ contains a $(-2)$-curve $C$ and $- K_Y$ is nef (e.g., $Y = \bb F_2$), we can find an irreducible rational nodal curve $D \in |-K_Y|$ so that $C \subset Y\setminus D$ as follows. First, $-K_X$ is nef (and hence base point free on a toric surface), we can take distinct general smooth members $D_1,D_2 \in |-K_Y|$ not contained in $C$ (and hence disjoint from $C$ as $K_Y \cdot C =0$). If we consider the pencil formed by $D_1$ and $D_2$ (whose generic member is an elliptic curve by the adjunction formula for arithmetic genus) and blow up finitely many base points, we obtain a rational elliptic surface, as we may assume base points are simple and hence exceptional divisors give sections. Now, since a generic singular fiber $D$ is of type $I_1$ (i.e., an irreducible rational nodal curve as $g_a(D) = 1$), we obtain a desired rational curve $D \in |-K_Y|$ by taking the pencil general enough, as singular fibers are characterized by vanishing of the discriminant and a fiber of type $I_1$ corresponds to a simple zero (see \cite{schutt2019elliptic} for an exposition of (rational) elliptic surfaces). In particular, $(Y,D)$ is a log Calabi-Yau surface with maximal boundary and $Y\setminus D$ is not quasi-affine as it contains a $(-2)$-curve.
\end{example}
We also have more examples of non-quasi-affine log Calabi-Yau surfaces with maximal boundary. 
\begin{example} Let $Y$ be a rational elliptic surface with a section and a singular fiber (e.g., the blow-up of $\bb P^2$ at the $9$ base points of a cubic pencil). Note in this case $-K_X = F$ where $F$ is a fiber class. In particular, if $|F|$ has a Kodaira fiber $D$ of type $I_n$ ($n \geq 1$), then $(Y,D)$ is a log Calabi-Yau surface with maximal boundary. Also note $U = Y\setminus D$ is not quasi-affine in this case since $U$ contains other fibers. 
\end{example}
\begin{notation}
    Let $(Y,D)$ be a log Calabi-Yau surface with maximal boundary. Within its complex deformation class, there exists a distinguished class in the sense that it induces a split mixed Hodge structure. See \cite{hacking2023homological}*{\href{https://arxiv.org/pdf/2005.05010}{Section 2.2.}} for more details. In the sequel, we assume that a log Calabi-Yau surface with maximal boundary is equipped with this distinguished complex structure. 
\end{notation}
In this setting, we have the followinghomological mirror symmetry (\cite{hacking2023homological}*{\href{https://arxiv.org/pdf/2005.05010}{Theorem 1.1.}}). 
\begin{theorem}[Hacking-Keating]\label{thm: hk hms}
    Let $(Y,D)$ be a log Calabi-Yau surface with maximal boundary (and with distinguished complex structure). Then, there is a Weinstein $4$-fold $M$ together with a Lefschetz fibration $w:M \to \bb C$ such that there are compatible equivalences
    \[
    \perf D = \D^\pi \operatorname{\sf{Fuk}}(\Sigma) \quad \perf Y \simeq \D^b\sf{Fuk}(w)\quad \perf U \simeq \D^b \cal W\sf{Fuk}(M),
    \]
    where $\sf{Fuk}(\Sigma)$ denotes the Fukaya category of a smooth fiber $\Sigma$ of $w$ near infinity, $\sf{Fuk}(w)$ denotes the directed Fukaya category of $w$ and $\cal W\sf{Fuk}(M)$ denotes the wrapped Fukaya category of $M$. 
\end{theorem}
Now, let us recall constructions of Lagrangian translations introduced in \cite{hacking2021symplectomorphisms}.
\begin{construction}[Lagrangian translations]
    Suppose $M$ is an exact symplectic $4$-manifold that is the total space of an almost toric fibration $$\pi:M \to B$$ only with focus-focus singularities (cf. e.g., \cite{symington2003four}*{\href{https://arxiv.org/pdf/math/0210033}{Definition 4.2.}}), where $B$ is an integral affine manifold homeomorphic to a disk. Further suppose there is a fixed Lagrangian section $L_0$ for $\pi$. Those settings are applicable when $M$ is a mirror to a log Calabi-Yau surface. Now, for the smooth locus $M^\sf{sm}$ of $\pi$, there are so called global action-angle coordinates $$\fr p_{L_0}:T^*B \to M^\sf{sm}$$ so that over the zero section of $T^*B$, it functions as the fixed Lagrangian section $L_0$. Roughly speaking, $\fr p_{L_0}$ takes $(q,p) \in T^*B$ to a point in $M^\sf{sm}$ given by transporting $L_0(q) \in M$ along the time-one Hamiltonian flow prescribed by $p \in T_q^*B$. Now, for a Lagrangian section $L:B \to M$, we define 
    \[
    \sigma_{L}:M^\sf{sm} \to M^\sf{sm}, \quad x \mapsto \fr p_{L_0}\l(\tilde x + ({L({\pi(x))}})^\sim\r),
    \]
    where $\tilde y \in T^*B$ denote any preimage of $y$ under $\fr p_{L_0}$. Intuitively, we are adding $L - L_0$ at each fiber using the linear structure. By \cite{hacking2021symplectomorphisms}*{\href{https://arxiv.org/pdf/2112.06797}{Proposition 4.11.}}, $\sigma_L$ is well-defined and extends to a symplectomorphism of $M$, which is called the \textbf{Lagrangian translation} by $L$ (with respect to $L_0$).      
\end{construction}
The following are the main results in \cite{hacking2021symplectomorphisms}*{\href{https://arxiv.org/pdf/2112.06797}{Corollary 4.21, Theorem 5.5}}.
\begin{theorem}[Hacking-Keating]\label{thm: hk lag tran}
    Let $(Y,D)$ be a log Calabi-Yau surface with maximal boundary and set $U:= Y\setminus D$. Define the subgroup $Q:= \{\ecal L \in \pic Y \mid \ecal L|_D \iso \ecal O_D\} \subset \pic Y$. Fix a toric model of $(Y,D)$ and consider the corresponding almost toric fibration $\pi:M \to B$ for the mirror partner $(M,w)$ (cf. \cite{hacking2023homological}*{\href{https://arxiv.org/pdf/2005.05010}{Section 1.1.}}). Further, fix a Lagrangian section $L_0$ of $\pi$. Then, the following assertions hold.
    \begin{enumerate}
        \item There exists a canonical diffeomorphism $\iota: U \to M$ (cf. \cite{hacking2021symplectomorphisms}*{\href{https://arxiv.org/pdf/2112.06797}{Lemma 4.20}}) that induces bijections
        \begin{center}\DisableQuotes
            \begin{tikzcd}
\text{$\pic(U)$} \arrow[r, "\sim"]  & \l\{\parbox{7cm}{Lagrangian sections of $\pi$ up to fiber-preserving Hamiltonian isotopy}\r\} =:\sf{Lag}_\pi               \\
Q \arrow[r, "\sim"]                   & \l\{\parbox{7cm}{Lagrangian sections of $\pi$ equal to $L_0$ near $\partial M$ up to fiber preserving Hamiltonian isotopy}\r\}=: \sf{Lag}_{\pi,L_0} 
\end{tikzcd}
        \end{center}
\item Take a line bundle $\ecal L \in Q$ and the corresponding Lagrangian section $L \in \sf{Lag}_{\pi,L_0}$ as above. Then, $\sigma_L$ induces an autoequivalence ${\sigma_L}_*$ of $\D^b\cal W\sf{Fuk}(M)$ and under homological mirror symmetry (Theorem \ref{thm: hk hms}), $\tau_{\ecal L|_U}$ corresponds to $\sigma_L$, i.e., we have a pt-equivalence
\[
(\perf U,\tau_{\ecal L|_U}) \simeq (\D^b \cal W\sf{Fuk}(M),{\sigma_L}_*). \qedhere
\]
\end{enumerate}
\end{theorem}
\begin{remark}\label{rem: compactification and serre functor} Use the same notations as in Theorem \ref{thm: hk lag tran}. The action of $Q$ on $\perf U$ only depends on the image $\bar Q \subset \pic U$ of $Q$ under the restriction $\pic Y \to \pic U$. We have $Q \iso \bar Q$ if and only if $D$ is negative definite or indefinite, i.e., the intersection form associated to $D_i$ is negative definite or indefinite when we write $D = D_1 + \cdots + D_n$ where $D_i$ are irreducible components that are necessarily rational curves (cf. \cite{hacking2021symplectomorphisms}*{\href{https://arxiv.org/pdf/2112.06797}{Section 2.3.}} for more discussions). In such a case, by \cite{hacking2021symplectomorphisms}*{\href{https://arxiv.org/pdf/2112.06797}{Remark 5.6.}}, we have a pt-equivalence 
    \[
    (\perf Y, \tau_\ecal L) \simeq (\D^b \sf{Fuk}(w), {\sigma_L}_*). \qedhere
    \]
\end{remark}
As a consequence, we get the following result. 
\begin{prop}\label{prop: reform hk}
    Use the same notations as in Theorem \ref{thm: hk lag tran}. Then, the diffeomorphism $\iota:U \to M$ and homological mirror symmetry induces a multi-pt-equivalence
    \[
    (\perf U, \tau_Q) \simeq (\D^b \cal W\sf{Fuk}(M), \tau_{\pi,\partial}), 
    \]
    where we set the multi-polarizations $\tau_Q$ and $\tau_{\pi,\partial}$ to be $$\tau_Q:= \{\tau_{\ecal L|_U} \mid \ecal L \in Q\} \subset \auteq \perf U, \quad \tau_{\pi,\partial}:= \{{\sigma_L}_*\mid L \in \sf{Lag}_{\pi,L_0}\} \subset \auteq \D^b\cal W\sf{Fuk}(M).$$ 
    In particular, the multi-polarization $\tau_{\pi,\partial}$ does not depend on the choice of $L_0$. 
\end{prop}
\begin{proof}
By Theorem \ref{thm: hk lag tran}, homological mirror symmetry induces the following commutative diagram:
\begin{center}\DisableQuotes
\begin{tikzcd}
Q \arrow[d, "\cong"'] \arrow[r, two heads]    & \tau_Q \arrow[r, hook] \arrow[d]      & \auteq \perf U \arrow[d, "\cong"] \\
{\mathsf{Lag}_{\pi,L_0}} \arrow[r, two heads] & {\tau_{\pi,\partial}} \arrow[r, hook] & \auteq \D^b\cal W\sf{Fuk}(M)     
\end{tikzcd}
\end{center}
which proves the first claim. Since homological mirror symmetry functor does not depend on the choice of $L_0$, the multi-polarization $\tau_{\pi,\partial}$ is determined as the image of $\tau_Q$ under homological mirror symmetry and does not depend on the choice of $L_0$. 
\end{proof}
As a further consequence, we can obtain a construction of a mirror algebraic partner as a pt-spectrum, which may be viewed as a partial converse of the main result of \cite{gross2015mirror}.
\begin{prop}\label{prop: mirror partner as spectra}
    With the same notations as in Theorem \ref{thm: hk lag tran}, we have
    \[
    U \iso  \spec^{\tau_{\pi,\partial}} \D^b \cal W\sf{Fuk}(M).
    \]
    Moreover, if there exists an $M$-compatible tt-structure $\tens_M$ on $\D^b \cal W\sf{Fuk}(M)$ such that its unit $\bb 1_M$ corresponds to the structure sheaf $\ecal O_U$ under homological mirror symmetry and that for any $L,L' \in \sf{Lag}_{\pi,L_0}$, we have 
    \[
    \sigma_L(\sigma_{L'}(\bb 1_M)) \iso \sigma_L(\bb 1_M) \tens_M \sigma_{L'}(\bb 1_M),
    \]
    then we have 
    \[
    U \iso  \spec^{\tau_{\pi,\partial}} \D^b \cal W\sf{Fuk}(M)\iso \spec_{\tens_M}\D^b \cal W\sf{Fuk}(M). \qedhere
    \]
\end{prop}
\begin{proof}
    First, we claim $\bar Q$ contains an ample line bundle. Indeed, take a line bundle $\ecal L$ on $Y$ that restricts to an ample line bundle on $U$ (which exists as the restriction $\pic Y \to \pic U$ is surjective). If $\ecal L|_D$ is not trivial on some irreducible components of the boundary divisor, then we can twist by divisors corresponding to such components so that the restriction to $D$ is trivial without changing the restriction to $U$. So, we may assume $\ecal L \in Q$ and hence we have an ample line bundle $\ecal L|_U \in \bar Q$. Now, since we have $\spec_{\tens_U} \perf U \subset \spec^{\tau_\ecal L} \perf U$ for any $\tau_{\ecal L} \in \tau_Q \subset \pic U$ and $\bar Q$ contains an ample line bundle on $U$, we get
    \[
    U \iso \spec^{\tau_Q} \perf U \iso \spec^{\tau_{\pi,\partial}} \D^b \cal W\sf{Fuk}(M).
    \]
    Moreover, since under homological mirror symmetry, we see that $\bra{\sigma (\bb 1_M)\mid \sigma \in \tau_{\pi,\partial}} = \D^\pi\cal W \sf{Fuk(M)}$ and hence by the same argument as in Theorem \ref{thm: divisorial reconstruction}, we have 
    \[\spec^{\tau_{\pi,\partial}} \D^b \cal W\sf{Fuk}(M)\iso \spec_{\tens_M}\D^b \cal W\sf{Fuk}(M).\qedhere \]
\end{proof}

\begin{remark}\label{rem: mirror consequences}
    Use the same notations as in Theorem \ref{thm: hk lag tran}.
    \begin{enumerate}
        \item If $U$ is quasi-affine, then we already have $\spec_\vartriangle \D^b \cal W\sf{Fuk}(M) \iso U$ by \cite{matsui2023triangular}*{\href{https://arxiv.org/pdf/2301.03168}{Corollary 4.7.}}. However, as we have seen in examples, there are plenty of cases where $U$ is not quasi-affine. 
        \item Recall from Remark \ref{rem: compactification and serre functor} (cf. \cite{hacking2021symplectomorphisms}*{\href{https://arxiv.org/pdf/2112.06797}{Remark 5.6, Remark 5.7.}})) that if $D$ is negative-definite or indefinite, then for any $\ecal L \in Q$ and the corresponding $L \in \sf{Lag}_{\pi,L_0}$, we have a pt-equivalence 
    \[
    (\perf Y, \tau_\ecal L) \simeq (\D^b \sf{Fuk}(w), {\sigma_L}_*). 
    \] Thus, we know that we have
    \[
    Y \iso \spec_{\tens_Y^\bb L}\perf Y \underset{\text{open}}{\subset} \spec^{\tau_Q} \perf Y \iso \spec^{\tau_{\pi,\partial}} \D^b  \sf{Fuk}(w),
    \]
    where $\tau_Q$ and $\tau_{\pi,\partial}$ are naturally extended to $\perf Y$ and $\D^b \cal W\sf{Fuk}(w)$. A natural question is if we can add more polarizations to $\tau_{\pi,\partial}$ to cut out $Y$ while only looking at the symplectic side. Note, by passing to $\D^b \sf{Fuk}(w) \simeq \perf Y$, we "compactify" $\D^b \cal W\sf{Fuk}(M)\simeq \perf U$ so that we have the Serre functor $\bb S_w = \bb S_Y$. In particular, we have a multi-polarization $\bar \tau_{\pi,\partial}:= \tau_{\pi,\partial}\cup \{\bb S_w\}$ on $\D^b  \sf{Fuk}(w)$. If an algebraic mirror $Y$ is a toric surface (or more generally a surface of general type), we conjecture $$Y \iso  \spec^{\bar \tau_{\pi,\partial}} \D^b  \sf{Fuk}(w)$$ 
    (if and) only if $\omega_Y = \ecal O_Y(-D)$ is pt-ample. In particular, unless the compactification is nice, we expect to need more polarizations to cut out the Balmer spectrum.  
        \item In general, suppose we have an $M$-compatible tensor triangulated structure $\tens_M$ on the Fukaya-like category constructed from an almost torus fibration. We expect we can apply similar strategies as above to show $\tens_M$ is geometrically equivalent to the tt-strctures $\tens_X^\bb L$ of its algebraic mirror $X$ if there are enough Lagrangian sections (so that they correspond to an ample family) and tensor product of those sections are given by certain modifications of fiber-wise additions. \qedhere
    \end{enumerate}
\end{remark}
Let us finish this subsection with plenty of natural questions arising from our observations. 
\begin{question}
    Let $M$ be a Weinstein $4$-fold and suppose there is an almost-torus fibration $\pi:M \to B$ with a fixed Lagrangian section $L_0$. 
    \begin{enumerate} 
        \item Can we characterize an almost toric fibration $\pi':M \to B$ for which we have an isomorphism 
        \[
        \spec^{\tau_{\pi,\partial}} \D^b \cal W\sf{Fuk}(M) \iso \spec^{\tau_{\pi',\partial}} \D^b \cal W\sf{Fuk}(M)?
        \]
        More generally, given a multi-polarization $\pic(X)$ on $\perf X \simeq \sf{Fuk}(M)$, can we construct the corresponding Lagrangian fibration of $M$? 
        \item If $\spec^{\tau_{\pi,\partial}} \D^b \cal W\sf{Fuk}(M)$ is a log Calabi-Yau surface with maximal boundary, then is it a mirror partner of $M$? In other words, can we predict mirror partners of $M$ using pt-spectra? 
    \item Can we give symplecto-geometric formulae for prime thick subcategories (or even prime $\tens$-ideals) in $\spec^{\tau_{\pi,\partial}} \D^b \cal W\sf{Fuk}(M)$?
    \item Using such formulae, can we construct $\spec^{\tau_{\pi,\partial}} \D^b \cal W\sf{Fuk}(M)$ more directly from $M$ without fully going through the wrapped Fukaya category? 
    \end{enumerate}
\end{question}

\appendix
\section{Functoriality of spectra and topological lattices}\label{append}
In Remark \ref{rem: not functorial}, we mentioned that the association taking pt-categories to pt-spectra does not define a functor due to lack of well-definedness on morphisms. To remedy this issue, the following is a key observation. 
\begin{lemma}\label{Appendix: key lemma}
    Let $F:(\cal T_1,\tau_1) \to (\cal T_2,\tau_2)$ be a pt-functor. Then, a map
    \[
    F\inv: \Th^{\tau_2} \cal T_2 \to \Th^{\tau_1} \cal T_1,\quad \cal I \mapsto F\inv(\cal I)
    \]
    is well-defined, continuous (with respect to the subspace topology of the Balmer topology) and preserving inclusions of categories. 
\end{lemma}
\begin{proof}
    First, for $\cal I \in \Th^{\tau_2} \cal T_2$, we have $F\inv(\cal I) \in \Th^{\tau_1} \cal T_1$ since 
    \[
        \tau_1\inv (F\inv(\cal I))=  F\inv(\tau_2\inv(\cal I)) =  F\inv(\cal I). 
        \]
        Hence, $F \inv$ is well-defined. To show continuity, take a collection $\ecal E$ of objects in $\cal T_1$. Then, we have
            \[
            (F\inv)\inv(V(\ecal E)) = \{\cal I \in \Th \cal T_2 \mid F\inv(\cal I) \cap \ecal E = \emp\}= \{\cal I \in \Th \cal T_2 \mid \cal I \cap F(\ecal E) = \emp\} = V(F(\ecal E)),
            \]
            i.e., $F\inv:\Th \cal T_2 \to \Th \cal T_1$ is continuous, which restricts to a continuous map $\Th^{\tau_2} \cal T_2 \to \Th^{\tau_1} \cal T_1$ since we put the subspace topology on $\Th^{\tau_i}\cal T_i$. The fact that $F\inv$ preserves inclusions is clear. 
\end{proof}
From here, there are essentially two approaches; one is to limit possible functors between pt-categories and the other is to modify pt-spectra. We will briefly discuss the former below and then move on to the latter.
\begin{definition}
    A pt-functor $F:(\cal T_1, \tau_1) \to (\cal T_2, \tau_2)$ is said to be \textbf{geometric} if the map
        \[
        \spc(F):\spc^{\tau_2}\cal T_2 \to \spc^{\tau_1} \cal T_1, \quad \cal P \mapsto F\inv(\cal P)
        \]
        is well-defined (in which case it is necessarily continuous). By Lemma \ref{Appendix: key lemma}, $F$ is geometric if and only if for any $\cal P \in \spc^{\tau_2} \cal T_2$, we have $F\inv(\cal P) \in \spc_\vartriangle \cal T_1$.
        
        By definition, we get a functor
        \[
        \spc:\sf{ptCat_\sf{geom}}^\op \to \sf{Top},
        \]
        where $\sf{ptCat_\sf{geom}}$ denotes the category of pt-categories with geometric pt-functors and $\sf{Top}$ denotes the category of topological spaces.
\end{definition}
\begin{example}
    If we have tt-categories $(\cal T_i,\tens_i)$ for $i=1,2$ with a $\tens$-ample object $A_1 \in \cal T_1$, then any essentially surjective tt-functor $F$ between $(\cal T_i,\tens_i)$ gives rise to a geometric pt-functor $F: (\cal T_1, - \otimes_1 A_1) \to (\cal T_2, - \otimes_2 F(A_1))$ as $F(A_1)$ is $\tens$-ample, but not the other way around in general. For example, if we take $\cal T_1 = \cal T_2 = \perf X$ for a smooth projective variety $X$ and take an ample line bundle $\ecal A$ and any line bundle $\ecal L$ on $X$, then $\tau_\ecal L:(\perf X,\tau_\ecal A)\simeq (\perf X, \tau_\ecal A)$ is a pt-equivalence but it is not a tt-functor with respect to $\tens_{\ecal O_X}^\bb L$.
\end{example}
Now, let us move on to modify pt-spectra to ensure functoriality while retaining enough information so that we can recover the original pt-spectra in certain cases of our interest. First, let us introduce the following notion as a target category. 
\begin{definition} Define a \textbf{topological complete lattice} to be a complete lattice $(L,\leq)$ together with topology defined on the set $L$, where no compatibility is required. Let 
\[
\sf{TopCLat}_\sf{meet}
\]
denote the category of topological complete lattices whose morphisms are poset maps that preserve arbitrary meets and that are continuous.  
\end{definition}
\begin{remark}
    Since a lattice can already be thought of as a generalization of topological spaces, a topological (complete) lattice can be thought of as essentially having two different topologies. Compatibility of two topologies indicates certain geometricity. See Remark \ref{rem: Hochster}. 
\end{remark}
\begin{construction}
    Let $\sf{ptCat}$ denote the category of pt-categories with pt-functors. Define the \textbf{thick spectrum functor} to be a functor
        \[
        \Th:\sf{ptCat}^\op \to \sf{TopCLat}_\sf{meet}
        \]
        sending a pt-category $(\cal T, \tau)$ to the sublattice $\Th^\tau \cal T \subset \Th \cal T$ of fixed points by $\tau$ equipped with the subspace topology of the Balmer topology on $\Th \cal T$ (cf. Notation \ref{notationa: prelim} (i) and (viii)) and sending a pt-functor $F:(\cal T_1,\tau_1) \to (\cal T_2,\tau_2)$ to the continuous poset map $F\inv:\Th^{\tau_2} \cal T_2 \to \Th^{\tau_1} \cal T_1$ (cf. Lemma \ref{Appendix: key lemma}). For well-definedness, we need to note the following two facts:
        \begin{itemize}
            \item For any non-empty subset $S \subset \Th^\tau \cal T$, we see that $\bigcap_{\cal P \in S} \cal P$ is a thick subcategory and that
            \[
            \tau\l(\bigcap_{\cal P \in S} \cal P\r) = \bigcap_{\cal P \in S} \tau(\cal P) = \bigcap_{\cal P \in S} \cal P,
            \]
            i.e., $\bigcap_{\cal P \in S} \cal P \in \Th^\tau \cal T$. Note we also have $\cal T \in \Th^\tau \cal T$, which is a unique largest element, i.e., the empty meet. Since a poset that admits an arbitrary meet is a complete lattice, we have $\Th^\tau \cal T \in \sf{TopCLat}_\sf{meet}$. 
            \item For any subset $S\subset \Th^{\tau_2}\cal T_2$, we have
            \[
            F\inv\l(\bigcap_{\cal P\in S}\cal P\r) = \bigcap_{\cal P \in S} F\inv(\cal P)
            \]
            and hence $F\inv$ is meet-preserving, i.e., we have shown \[F \inv \in \hom_{\sf{TopCLat}_{\sf{meet}}}(\Th^{\tau_2}\cal T_2,\Th^{\tau_1}\cal T_1). \qedhere\]
        \end{itemize}
\end{construction}
Note that for pt-categories of algebro-geometric interests, we can indeed recover the pt-spectrum $\spc^\tau \cal T$ from the topological complete lattice structure of $\Th^\tau \cal T$ by the following result, so we are not losing too much information by considering thick spectra in such cases.
\begin{prop}
    Let $X$ be a noetherian scheme and let $\ecal L$ be a $\tens$-ample line bundle. Then, we have
    \[
    \spc^{\tau_\ecal L} \perf X = \{\cal P \in \Th^{\tau_\ecal L} \perf X \mid \textrm{the poset $\{\cal Q \in \Th^{\tau_\ecal L} \cal T \mid \cal P \subsetneq \cal Q\}$ has a unique smallest element}\} 
    \]
    where we put the subspace topology from $\Th^{\tau_\ecal L} \perf X$ on the right hand side.  
\end{prop}
\begin{proof}
    By the same arguments as Proposition \ref{example:guiding}, $\Th^{\tau_\ecal L} \perf X$ agrees with the lattice of radical thick $\tens$-ideals of $(\perf X, \tens_X)$, noting $(\perf X, \tens_X)$ is a rigid tt-category (cf. \cite{Balmer_2005}*{\href{https://www.math.ucla.edu/~balmer/Pubfile/Spectrum.pdf}{Proposition 4.4.}}). Now, by \cite{Matsui_2021}*{Proposition 4.7}, the right hand side agrees with the Balmer spectrum and hence with the pt-spectrum $\spc^{\tau_\ecal L}\perf X$ as sets and their topologies agree as both of them are the subspace topology from $\Th \perf X$. 
\end{proof}
Note exactly the same proof works for any rigid tt-category with Noetherian tt-spectrum and $\tens$-ample objects. A priori the right hand side may differ from the pt-spectrum in general and may be easier to compute as it can be thought of as one variant of relative Matsui spectra as discussed in \cite{matsukawa2025spectrum}. 
\begin{remark}  \label{rem: Hochster}
For a noetherian scheme $X$ with a $\tens$-ample line bundle $\ecal L$, note that by the Hochster duality it is enough to remember the lattice structure on $\Th^{\tau_\ecal L}\perf X$ in order to recover the Balmer topology on $\spec^{\tau_\ecal L} \perf X$ (cf. \cite{kock2017hochster}), but in general we expect the Balmer topology on $\Th^\tau \cal T$ is extra information.

It is also natural to ask how much the converse holds. Namely, if topology on $\spc^{\tau_\ecal L} \perf X$ constructed by the Hochster duality, if applicable, agrees with the Balmer topology, then is it isomorphic to (a Fourier--Mukai partner of) $X$?
\end{remark}

\bibliography{bib}
\end{document}

%% file: name.tex
\author{Daigo Ito} 
\date{}

%% file: style.tex
\usepackage{stix} 
\usepackage[english]{babel} 
\usepackage{amsmath,amsfonts,amsthm,amscd}
\usepackage{ascmac}
\usepackage{appendix}
\usepackage{bm}
\usepackage{cases}
\usepackage{changepage}
\usepackage{enumitem}
\usepackage{etoolbox}
\usepackage{float}
\usepackage{geometry}
\usepackage{graphicx}
\usepackage[utf8]{inputenc}
\usepackage [autostyle, english = american]{csquotes}
\usepackage{scalerel}
\usepackage{ytableau}
\usepackage{tikz-cd} 
\usepackage{tocloft}
\usepackage{setspace}
\usepackage[color,all]{xy}
\usepackage[cal=euler]{mathalfa}
\usepackage{url}
\usepackage{CJKutf8}
\makeatletter
\newcommand{\address}[1]{\gdef\@address{#1}}
\newcommand{\email}[1]{\gdef\@email{\url{#1}}}
\newcommand{\website}[1]{\gdef\@website{\url{#1}}}
\newcommand{\@endstuff}{\par\vspace{\baselineskip}\noindent\small
\begin{tabular}{@{}l}\scshape{Daigo Ito} \\ \scshape\@address\\\textrm{E-mail address:} \@email \\\textrm{Website:} \@website\end{tabular}}
\AtEndDocument{\@endstuff}
\makeatother
\address{Department of Mathematics, University of California, Berkeley, Evans Hall 935, CA 94720-3840}
\email{daigoi@berkeley.edu}
\website{https://daigoi.github.io/}

\usepackage[alphabetic,backrefs]{amsrefs}
\usepackage{hyperref}
\hypersetup{colorlinks,linkcolor=blue,citecolor=red}

\DeclareMathOperator{\aut}{Aut}

\DeclareMathOperator{\auteq}{Auteq}

\DeclareMathOperator{\End}{End}

\DeclareMathOperator{\FM}{FM}

\let \hom \relax
\DeclareMathOperator{\hom}{Hom}

\DeclareMathOperator{\ob}{Ob}

\DeclareMathOperator{\perf}{\mathsf{Perf}}

\DeclareMathOperator{\pic}{Pic}

\DeclareMathOperator{\proj}{Proj}

\DeclareMathOperator{\rank}{rank}

\DeclareMathOperator{\spc}{Spc}
\DeclareMathOperator{\spec}{Spec}

\let \sl \relax 
\DeclareMathOperator{\sl}{SL}

\DeclareMathOperator{\supp}{Supp}

\DeclareMathOperator{\Th}{Th}

\newcommand {\bb}{\mathbb}
\renewcommand {\cal}{\mathcal}
\newcommand{\ecal}{\mathscr}
\renewcommand {\epsilon}{\varepsilon}
\newcommand {\fr}{\mathfrak}

\newcommand {\bra}[1]{\langle{#1}\rangle}
\renewcommand {\l}{\left}

\renewcommand {\r}{\right}

\newcommand*{\DashedArrow}[1][]{\mathbin{\tikz [baseline=-0.25ex,-latex, dashed,#1] \draw [#1] (0pt,0.5ex) -- (1.3em,0.5ex);}}
\newcommand {\ratmap}{\DashedArrow[->,densely dashed    ]} 

\newcommand {\emp}{\emptyset}

\newcommand {\inj}{\hookrightarrow}
\newcommand {\inv}{^{-1}}
\newcommand {\iso}{\cong}

\newcommand {\surj}{\twoheadrightarrow}
\newcommand {\tens}{\otimes}

\newsavebox{\pullbacks}
\sbox\pullbacks{%
\begin{tikzpicture}%
\draw (0,0) -- (1ex,0ex);%
\draw (1ex,0ex) -- (1ex,1ex);%
\end{tikzpicture}}

\newcommand{\coh}{\mathsf{coh}}

\newcommand{\D}{\mathsf{D}}

\newcommand{\h}{{\mathrm{H}}}

\newcommand{\ho}{\mathsf{Ho}}

\newcommand {\id}{{\rm id}}

\newcommand {\op}{\mathsf{op}}

\newcommand{\ser}{\mathsf{Ser}}
 
\let \sf \relax 
\newcommand{\sf}{\mathsf}

\newcommand{\injar}{\ar@{^(->}} 
\newcommand{\prarrow}[2]{\ar@<0.5ex>[r]^-{#1} \ar@<-0.5ex>[r]_-{#2}}
\newcommand{\plarrow}[2]{\ar@<0.5ex>[l]^-{#1} \ar@<-0.5ex>[l]_-{#2}}
\newcommand{\pdarrow}[2]{\ar@<0.5ex>[d]^-{#1} \ar@<-0.5ex>[d]_-{#2}}
\newcommand{\puarrow}[2]{\ar@<0.5ex>[u]^-{#1} \ar@<-0.5ex>[u]_-{#2}}

\theoremstyle{plain}
\newtheorem{theorems}{Theorem}[section] 
\newtheorem{claims}{Claim}[theorems]
\newtheorem{conjectures}[theorems]{Conjecture}
\newtheorem{corollaries}[theorems]{Corollary}
\newtheorem{lemmas}[theorems]{Lemma} 
\newtheorem{props}[theorems]{Proposition}

\newtheorem{penmdef}[claims]{Definition}

\theoremstyle{definition}
\newtheorem{constructions}[theorems]{Construction}
\newtheorem{definitions}[theorems]{Definition} 
\newtheorem{axioms}[theorems]{Axiom} 
\newtheorem{examples}[theorems]{Example}     
\newtheorem{notations}[theorems]{Notation}         

\newtheorem{question}[theorems]{Question}
\newtheorem{obss}[theorems]{Observation}
\newtheorem{penmlem}[claims]{Lemma}
\newtheorem{penmthm}[claims]{Theorem}
\newtheorem{penmcor}[claims]{Corollary}
\newtheorem{penmeg}[claims]{Example} 
\newtheorem{inclaims}[claims]{Claim}

\theoremstyle{remark}
\newtheorem{remarks}[theorems]{Remark}

\newtheorem{penmrem}[claims]{Remark}

\makeatletter
\newtheoremstyle{indented}
  {1pt}
  {1pt}
  {\addtolength{\@totalleftmargin}{1.5em}
   \addtolength{\linewidth}{-1.5em}
   \parshape 1 1.5em \linewidth}
  {}
  {\bfseries}
  {.}
  {.5em}
  {}
\makeatother

\theoremstyle{indented}
\newtheorem{pinddef}[claims]{Definition} 
\newtheorem{pindlem}[claims]{Lemma}
\newtheorem{pindthm}[claims]{Theorem}
\newtheorem{pindcor}[claims]{Corollary}
\newtheorem{pindeg}[claims]{Example} 
\newtheorem{pindrem}[claims]{Remark} 
\newtheorem{pindq}[claims]{Question} 
\newtheorem{pindc}[claims]{Conjecture} 
\newtheorem{pindclaim}[claims]{Claim}

\newenvironment{theorem}
{
	\pushQED{\qed}\begin{theorems}}
	{\popQED\end{theorems}}

\newenvironment{prop}
{
	\pushQED{\qed}\begin{props}}
	{\popQED\end{props}}

\newenvironment{notation}
{
	\pushQED{\qed}\begin{notations}}
	{\popQED\end{notations}}

\newenvironment{corollary}
{
	\pushQED{\qed}\begin{corollaries}}
	{\popQED\end{corollaries}}

\newenvironment{definition}
{
	\pushQED{\qed}\begin{definitions}}
	{\popQED\end{definitions}}

\newenvironment{lemma}
{
	\pushQED{\qed}\begin{lemmas}}
	{\popQED\end{lemmas}}

\newenvironment{remark}
{
	\pushQED{\qed}\begin{remarks}}
	{\popQED\end{remarks}}
	
\newenvironment{example}
{
	\pushQED{\qed}\begin{examples}}
	{\popQED\end{examples}}

\newenvironment{construction}
{
	\pushQED{\qed}\begin{constructions}}
	{\popQED\end{constructions}}

\newenvironment{enmlem}
{
	\pushQED{\qed} \begin{penmlem}}
	{\popQED\end{penmlem}}

\MakeOuterQuote{"}